\documentclass[11pt]{amsart}
\usepackage{amsmath, amsfonts, amssymb,amsthm, amscd}

\usepackage{latexsym,graphics}
\usepackage{epsfig,epsf,psfrag,color}

\usepackage{graphicx}
\usepackage{epsf}
\usepackage{epsfig}
\usepackage{verbatim}
\usepackage{psfrag}
\usepackage{color}
\usepackage{hyperref}

\definecolor{blue}{rgb}{0,0,1}
\definecolor{red}{rgb}{1,0,0}
\definecolor{green}{rgb}{0,.7,0}
\newtheorem{theorem}{Theorem}

\newtheorem{``theorem''}[theorem]{``Theorem''}
\newtheorem{lemma}[theorem]{Lemma}
\newtheorem{proposition}[theorem]{Proposition}
\newtheorem{corollary}[theorem]{Corollary}

\theoremstyle{definition}
\newtheorem{definition}[theorem]{Definition}

\theoremstyle{remark}
\newtheorem{remark}[]{Remark}

\newtheoremstyle{thm}
  {12pt}
  {12pt}
  {\itshape}
  {\parindent}
  {\scshape}
  {.}
  {5pt}
  {}

\theoremstyle{thm}
\newtheorem*{T6.17}{Theorem \ref{gradientthm1}}

\newtheoremstyle{prop}
  {12pt}
  {12pt}
  {\itshape}
  {\parindent}
  {\scshape}
  {.}
  {5pt}
  {}

\theoremstyle{prop}

\newtheoremstyle{lem}
  {12pt}
  {12pt}
  {\itshape}
  {\parindent}
  {\scshape}
  {.}
  {5pt}
  {}

\theoremstyle{lem}
\newtheorem*{L2.6}{Lemma \ref{lem2.4}}

\newtheoremstyle{defn}
  {12pt}
  {12pt}
  {\itshape}
  {\parindent}
  {\scshape}
  {.}
  {5pt}
  {}

\theoremstyle{defn}

\newtheoremstyle{examp}
  {12pt}
  {12pt}
  {}
   {\parindent}
  {\scshape}
  {.}
  {5pt}
  {}

\theoremstyle{examp}

\newtheoremstyle{cor}
  {12pt}
  {12pt}
  {\itshape}
  {\parindent}
  {\scshape}
  {.}
  {5pt}
  {}

\theoremstyle{cor}

\newtheoremstyle{recipe}
  {12pt}
  {12pt}
  {\itshape}
   {\parindent}
  {\scshape}
  {.}
  {5pt}
  {}

\theoremstyle{recipe}

\newtheoremstyle{rem}
  {12pt}
  {12pt}
  {}
   {\parindent}
  {\scshape}
  {.}
  {5pt}
  {}

\theoremstyle{rem}

\newtheoremstyle{quest}
  {12pt}
  {12pt}
  {\itshape}
  {\parindent}
  {\scshape}
  {.}
  {5pt}
  {}

\theoremstyle{quest}

\newcommand{\bp}{\begin{proof}}
\newcommand{\ep}{\end{proof}}

\renewcommand{\epsilon}{\varepsilon}

\newcommand{\Z}{{\mathbb Z}}

\newcommand{\id}{\operatorname{id}}

\newcommand{\loc}{\text{loc}}

\newcommand{\R}{{\mathbb R}}

\providecommand{\ker}[1]{$\text{ker}\ {#1}$}

\newcommand{\N}{{\mathbb N}}

\newcommand{\T}{\mathcal{T}}
\newcommand{\Rot}{\operatorname{Rot}}
\newcommand{\rot}{\operatorname{rot}}

\newcommand{\F}{\mathcal{F}}
\newcommand{\Diff}{\operatorname{Diff}}

\newcommand{\varphitilde}{\tilde{\varphi}}
\renewcommand{\u}{\tilde{u}}
\renewcommand{\v}{\tilde{v}}

\newcommand{\C}{\mathbb{C}}

\renewcommand{\loc}{\textup{loc}}

\newcommand{\Ftilde}{\tilde{\mathcal{F}}}
\newcommand{\ftilde}{\tilde{f}}

\renewcommand{\max}{\textup{max}}

\renewcommand{\S}{\mathcal{S}}

\newcommand{\half}{\frac{1}{2}}

\newcommand{\Jhat}{\hat{J}}

\newcommand{\ubar}{\bar{u}}

\newcommand{\M}{\mathcal{M}}

\renewcommand{\H}{\mathcal{H}}

\renewcommand{\H}{\mathbb{H}}

\renewcommand{\S}{\mathcal{S}}
\renewcommand{\u}{\tilde{u}}
\renewcommand{\v}{\tilde{v}}

\newcommand{\Ubar}{\bar{U}}

\newcommand{\gammadot}{\dot{\gamma}}

\newcommand{\Ham}{\mathcal{H}}
\newcommand{\varphihat}{\hat{\varphi}}

\newcommand{\Vbar}{\bar{V}}

\newcommand{\A}{\mathbb{A}}
\newcommand{\Atilde}{\tilde{\A}}
\newcommand{\Homeo}{\textup{Homeo}}
\newcommand{\Act}{\mathbf{A}}
\newcommand{\Fbar}{\overline{F}}
\newcommand{\Diophantine}{\mathcal{D}}

\newcommand{\vertical}{\textup{vert}}
\newcommand{\Zhat}{\hat{Z}}
\newcommand{\ghat}{\hat{g}}

\renewcommand{\Ddot}{\mathring{D}}

{\catcode`"=12 \gdef\hex{"}}

\mathchardef\laplace=\hex0001

\mathchardef\nabla=\hex0272

\makeatletter

\def\@@dalembert#1#2{\setbox0\hbox{$#1\mathrm I$}

  \vrule height\ht0 depth\z@ width.04\ht0

  \rlap{\vrule height\ht0 depth-.96\ht0 width.8\ht0}

  \vrule height.1\ht0 depth\z@ width.8\ht0

  \vrule height\ht0 depth\z@ width.1\ht0 }

\def\dalembert{\mathbin{\mathpalette\@@dalembert{}}\,}

\makeatother

\date{}

\begin{document}
\address{
    Barney Bramham\\
    Institute for Advanced Study, Princeton, NJ 08540}
\email{bramham@ias.edu}

\title[Periodic approximations of irrational pseudo-rotations]
{Periodic approximations of irrational pseudo-rotations using pseudoholomorphic curves}

\author[Barney Bramham]{Barney Bramham}
\date{\today}
\maketitle

\begin{abstract}
 We prove that every $C^\infty$-smooth, area preserving diffeomorphism of
 the closed $2$-disk having not more than one periodic point is the uniform limit
 of periodic $C^\infty$-smooth diffeomorphisms.
 In particular every smooth irrational pseudo-rotation can be $C^0$-approximated by integrable
 systems.  This partially answers a long standing question of A. Katok regarding zero entropy
 Hamiltonian systems in low dimensions.  Our approach uses pseudoholomorphic curve techniques from
 symplectic geometry.
\end{abstract}

\tableofcontents
\section{Introduction}

\subsection{The main result}
In this paper we prove a statement which is a significant first step towards answering the
following general question of Katok.

\begin{quote}
 \textit{In low dimensions is every conservative dynamical system with zero topological entropy
 a limit of integrable systems?}
\end{quote}

This is stated as problem 1 in \cite{K_elliptic_questions}, but relates also to the much older 
\cite{AnosovKatok}.
Low dimensions means maps on surfaces or flows on $3$-dimensional manifolds.  Arguably
one of the main obstacles to answering this question
affirmatively is the existence of ergodic components of positive measure.
Indeed, ergodic maps (with respect to Lebesgue measure) exhibit strongly different
dynamical behavior to integrable ones.  In particular, almost
every point is the initial condition for a dense orbit.

Using finite energy foliations by pseudoholomorphic curves we obtain a result of this nature
for the class of area preserving diffeomorphisms of the $2$-disk known as
irrational pseudo-rotations.  Irrational pseudo-rotations are of particular interest with regard to this 
question because they include all known ergodic disk maps with zero topological entropy. 

For each $t\in\R$ let $R_{2\pi t}:D\rightarrow D$ denote the rigid
rotation through angle $2\pi t$ on the disk $D=\{(x,y)\,|\, x^2+y^2\leq1\}$.  Let $\Diff^\infty_+(D)$ and 
$\Diff^\infty(D,\omega_0)$ be the spaces of $C^\infty$-smooth diffeomorphisms of the disk which 
preserve orientation and the area form $\omega_0=dx\wedge dy$ respectively.  

\begin{theorem}[Main result]\label{T:main_result_intro}
Suppose $\varphi\in\Diff^\infty(D,\omega_0)$ fixes the origin and has no other periodic points.
Then it is the $C^0$-limit of a sequence of maps of the form $\varphi_k=g_k^{-1}R_{2\pi p_k/q_k}g_k$,
for a sequence of conjugating maps $g_k\in\Diff^\infty_+(D)$ which fix the origin, and a sequence of
rationals $p_k/q_k$ converging to an irrational number.
\end{theorem}

Elements $\varphi\in\Diff^{\infty}(D,\omega_{0})$ having precisely one periodic point are
also known as (smooth) \emph{irrational pseudo-rotations}.  See definition \ref{D:irr_pseudorotation}.

Concerning the convergence of the approximation maps in theorem \ref{T:main_result_intro}
it is natural to ask whether $C^0$-convergence is so weak as to allow ``almost anything'' to be
obtained in the limit?  A more natural topology to consider Katok's question, as discussed in
\cite{K_elliptic_questions}, is at least a $C^{1,\varepsilon}$-topology, $\varepsilon>0$,
in which the topological entropy is lower semi-continuous.  The author would therefore like to thank
Patrice Le Calvez for pointing out the following and
the idea of its proof using work of Franks.  A slightly more general statement is proven in appendix A.

\begin{proposition}\label{P:LeCalvez_observation_intro}
Suppose $\varphi\in\Diff^\infty(D,\omega_0)$ is the $C^0$-limit of a sequence of maps of the form
$\varphi_k=g_k^{-1}R_{2\pi p_k/q_k}g_k$, for a sequence of conjugating maps $g_k\in\Diff^{\infty}_+(D)$ 
fixing the origin, and a sequence of rationals $p_k/q_k$ converging to an irrational number.
Then $\varphi$ necessarily has precisely one periodic point.  In particular it is an irrational
pseudo-rotation.
\end{proposition}

Thus $C^0$-convergence is still strong
enough to guarantee that the limit object is an irrational pseudo-rotation.  In particular
that it has zero entropy.  (Due, for example, to
Katok's theorem for $C^{1,\varepsilon}$ surface diffeomorphisms
\cite{Katok_entropy_and_horseshoes} that bounds entropy from above by the exponential growth rate
of periodic points.)

The existence of ergodic disk maps with zero entropy was established back in 1970 by Anosov and Katok
\cite{AnosovKatok}.  Previous to their constructions it was even an open question in the
non-conservative setting; Shnirelman 1930  \cite{Shnirelman} found a (non-area preserving) diffeomorphism
of the disk with a dense orbit, further discussion of which can be found in \cite{FayadKatok}.

It is interesting to compare the Anosov-Katok construction to the statement
of theorem \ref{T:main_result_intro}.  They construct an ergodic map $\varphi:D\rightarrow D$ as
the $C^\infty$-limit
of a sequence of maps $\varphi_k:D\rightarrow D$ which are inductively constructed with the following form.
For each $k\in\N$ there exists $(p_k,q_k)\in\Z\times\N$ relatively prime, and
$g_k\in\Diff^{\infty}(D,\omega_0)$, also fixing the origin, so that
\begin{equation}\label{E:AK_approximation_maps}
		      \varphi_k=g_k^{-1}\circ R_{2\pi p_k/q_k}\circ g_k.
\end{equation}
The maps $g_k$ are arranged so that the orbits of $\varphi_k$ increasingly spread
out over the disk as $k\rightarrow\infty$.  Consequently, the sequence $\{g_k\}$ blows up in every
$C^r$ topology.  But by interatively choosing $q_{k+1}-q_k$ sufficiently
large depending on the size of $\|g_k\|_{C^k}$, the $C^k$ norm
of $\varphi_k$ can be controlled.  A limiting subseqence
converges to a map $\varphi$ with the desired ``pathalogical'' behavior
such as a dense orbit, or ergodicity, or even weak mixing.  More details of this method, 
other results and questions, are in Fayad-Katok \cite{FayadKatok}.  See also 
Fayad-Saprykina \cite{Fayad_Sap}. 

In some sense then, theorem \ref{T:main_result_intro} reverses the limiting process just described
above.  However our conclusions are in two respects weaker than a word for word
converse to the Anosov-Katok construction.  Firstly, the convergence $\varphi_{k}\rightarrow\varphi$ in
(\ref{E:AK_approximation_maps}) is in the $C^{\infty}$-sense.  Secondly, each  $\varphi_{k}$ in
(\ref{E:AK_approximation_maps}) preserves the standard area form.  We do not show this for
the approximation maps in theorem \ref{T:main_result_intro}.  This raises natural questions for
further investigation.

A remark on integrability:  The notion of integrability for a map on a surface that
appears to be referred to in \cite{K_elliptic_questions}
is that the map should admit a ``first integral'', that is, a continuous or smooth real valued function
on the surface that is not constant on any open set but is constant on the orbits of the given map.
It is obviously in this sense that each of our approximation maps in theorem \ref{T:main_result_intro} is
integrable.   A natural question
is whether approximation maps can be found that are integrable in the Liouville-Arnold sense.  This
would follow if they were area preserving.

\subsection{Idea of the proof}
Pick a closed loop of Hamiltonians $H_{t}:D\rightarrow D$, over $t\in\R/\Z$, which generate 
a symplectic isotopy whose time-one
map is $\varphi$.  Denote the $1$-periodic path of Hamiltonian vector fields on
the disk by $X_{H_{t}}$.  For each $n\in\N$ equip $Z_{n}=\R/n\Z\times D$ with coordinates
$(\tau,z)$.  Then the vector field $R_{n}(\tau,z)=\partial_{\tau}+X_{H_{t}}(z)$ defines a
flow on $Z_{n}$ with time-one map $\varphi$ and first return map $\varphi^{n}$.

Consider the $4$-manifold $W_{n}:=\R\times Z_{n}$ with the unique almost complex structure satisfying
\[
 \left\{\begin{aligned}
		&J_n\partial_{\R}=R_n \\
		&J_n|_{TD}=i
\end{aligned}\right.
\]
where $\partial_{\R}$ is the vector field dual to the $\R$-coordinate on $W_{n}$.  Then $(W_{n},J_{n})$
is a so called cylindrical, symmetric, almost complex manifold.  That is, it is compatible in a precise
way with the necessary symplectic structures for the compactness framework from symplectic field theory
\cite{BEHWZ_SFT_compactness} to apply to $J_{n}$-holomorphic curves.  

In section \ref{S:pf_of_fef} 
we adapt techniques developed by Hofer, Wysocki, and Zehnder 
\cite{HWZ_tight,HWZ_convex,HWZ_tightII,HWZ_fef,HWZ_survey} to construct 
finite energy foliations of $(W_{n},J_{n})$.  That is, a foliation $\F$ by embedded surfaces 
that are the images of finite 
energy $J_{n}$-holomorphic curves, and where the set $\F$ is invariant under translations
in the $\R$-direction on $W_{n}$.

\begin{figure}[hbt]
\begin{center}
\includegraphics[scale=.45]{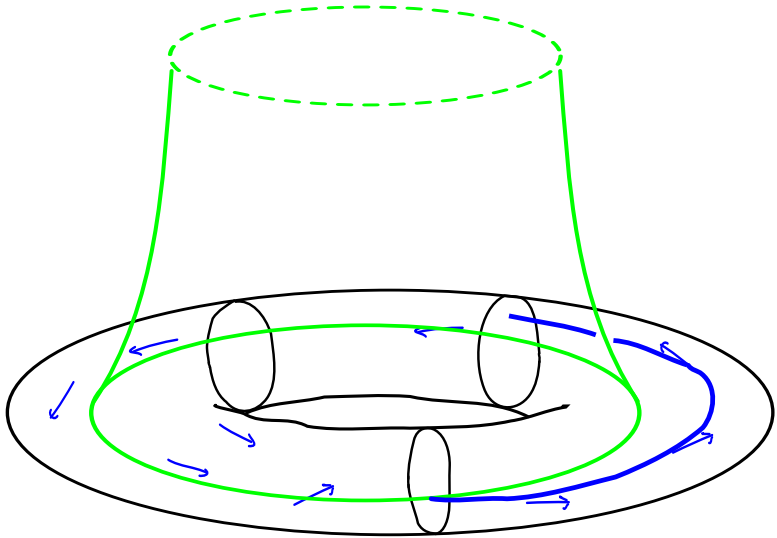}
\end{center}
\caption{A leaf $F$ in a finite energy foliation $\F$ of $W_{n}=\R\times Z_{n}$
intersects the hypersurface $\{0\}\times Z_{n}$ in a closed $1$-dimensional curve.
We use this to define a disk map $\varphi_{\F}$. }
\label{F:disk_map}
\end{figure}

In this setting a finite energy foliation $\F$ of $(W_{n},J_{n})$ can be used to define a disk map
$\varphi_{\F}:D\rightarrow D$ in the following fairly natural way: identify the mapping torus
$Z_{n}$ with the hypersurface $\{0\}\times Z_{n}\subset W_{n}$.  Then each leaf $F\in\F$
is either disjoint from $Z_{n}$ or intersects it transversally in an embedded closed curve as
in figure \ref{F:disk_map}.  The closed curve will then intersect each disk slice $\{\tau\}\times D$
in $Z_{n}$ transversally and precisely once (each such disk is also $J_{n}$-holomorphic).
If the curve intersects the disk $\{0\}\times D$ at a point $(0,\xi)$ and intersects the disk
$\{1\}\times D$ at a point $(1,\xi')$, then we set $\varphi_{\F}(\xi)=\xi'$.

For each $n\in\N$ there is also a ``trivial'' $\R$-invariant foliation of $(W_{n},J_{n})$ by
$J_{n}$-holomorphic curves we could refer to as the \emph{vertical} foliation $\F^{\vertical}(W_{n},J_{n})$.
The leaves of the vertical foliation are all of the form $\R\times\gamma(\R)$ for each
trajectory $\gamma:\R\rightarrow Z_{n}$ of the vector field $R_{n}$.  Usually this foliation is of
little interest due to its ``instability'' and each leaf besides one has infinite so called
$\lambda$-energy.  However if we use the vertical foliation to define a disk map as above, we 
obtain the map $\varphi$ of interest.
The leaves of the vertical foliation are characterized by having vanishing so called 
$\omega$-energy.  

It turns out that there exists a sequence of finite energy foliations 
$\F_{n}$ of $(W_{n},J_{n})$ with the following property: each leaf in $\F_{n}$ has 
$\omega$-energy zero or  
\[
						\{n\alpha\}\in[0,1)					
\]
where $\alpha\in\R$ can be identified with the rotation number of the circle map 
$\varphi|_{\partial D}$.  In particular $\alpha$ is irrational.  $\{x\}$ denotes the fractional part 
of $x$.  

Restricting to a subsequence $n_{j}$ for which $\{n_{j}\alpha\}\rightarrow0$ as 
$j\rightarrow\infty$, the leaves in $\F_{n_{j}}$ converge in some sense to the vertical foliation 
$\F^{\vertical}$.  Correspondingly the disk maps $\varphi_{\F_{n_{j}}}$ converge pointwise 
to the disk map $\varphi_{\F^{\vertical}}=\varphi$.  

By being slightly more careful one obtains uniform convergence.
Finally, each foliation $\F_{n}$ has a certain symmetry that ensures that the induced map
$\varphi_{\F_{n}}$ has $n$-th iterate $(\varphi_{\F_{n}})^{n}=\id_{D}$.  This is equivalent 
to $\varphi_{\F_{n}}$ being conjugate to a rigid rotation
through an angle $2\pi p/n$ for some $p\in\{0,1,\ldots,n-1\}$.  

\begin{remark}
We only use that $\varphi$ is an irrational pseudo-rotation at two points in this argument.  
First, to find the nice expression for the $\omega$-energy of the leaves.  Secondly to show that 
the maps $\varphi_{\F_{n}}$ are roots of unity.  Indeed the finite energy foliations themselves exist for generic area preserving disk maps, not just pseudo-rotations.  This will be shown in \cite{Bramham}.  
\end{remark}

\subsection{Results in the literature of related interest}
Using quite different techniques, Le Calvez proved in 2004, see Theorem 1.9 in \cite{LeCalvez_unlinked},
that every minimal $C^1$-diffeomorphism of the $2$-torus that is homotopic to the identity can be
$C^0$-approximated by periodic $C^1$-diffeomorphisms.  Recall that a diffeomorphism is minimal if every
point is the initial condition for a dense orbit.  Thus, in this result also strongly non-integrable maps
are approximated by, in some sense, integrable ones.

An interesting result about irrational pseudo-rotations in the class of homeomorphisms of the open and
closed annulus homotopic to the identity, was obtained by B\'eguin-Crovisier-LeRoux-Patou
\cite{LeRoux_closed_annulus} and B\'eguin-Crovisier-LeRoux \cite{LeRoux_open_annulus}.
Stated for maps on the closed disk this is as follows.  Let
$\varphi$ be an orientation preserving, measure preserving, homeomorphism of the disk with a single
periodic point and boundary rotation number $\alpha\in\R/\Z$.  Then the rigid rotation $R_{2\pi\alpha}$
is the $C^0$-limit of maps (not necessarily area preserving) conjugate to $\varphi$.
The authors of these papers note that one does not know from their
approach that $\varphi$ is in the closure of the set of maps conjugate to $R_{2\pi\alpha}$.

\subsection{Acknowledgements}
I would particularly like to thank Helmut Hofer for his interest, encouragement, and many
valuable discussions and suggestions.  Also for constructive comments on an earlier version of this 
paper.  I thank Patrice LeCalvez for proposition \ref{P:LeCalvez_observation_intro}, and 
Richard Siefring and Chris Wendl
for many helpful discussions about pseudoholomorphic curves.
I also wish to thank Anatole Katok for his interest.  

This work is based upon work supported by the National Science Foundation under agreement
No.\ DMS-0635607.  Any opinions, findings and conclusions or recommendations in this
material are those of the author and do not necessarily reflect the views of the National
Science Foundation.

\section{Irrational pseudo-rotations}\label{S:pseudo-rotations}

\begin{definition}\label{D:irr_pseudorotation}
A (smooth) \emph{irrational, pseudo-rotation} is a $C^\infty$-diffeomorphism
$\varphi:D\rightarrow D$ of the closed $2$-disk $D$
with the following properties: (1) $\varphi$ preserves the volume form $dx\wedge dy$.
(2) $\varphi(0)=0$.  (3) $\varphi$ has no periodic points in $D\backslash\{0\}$.
\end{definition}

There are equivalent definitions which admit generalizations to rational pseudo-rotations
which we will not need.  See for example \cite{LeCalvez_survey} and \cite{LeRoux_closed_annulus}.

If $\varphi:D\rightarrow D$ is an irrational pseudo-rotation, then the restriction of
$\varphi$ to the boundary is an orientation preserving circle diffeomorphism without periodic
points.  It therefore has irrational rotation number on the boundary.

More precisely, let $\pi:\R\rightarrow\partial D$ be the projection map $x\mapsto e^{2\pi ix}$.
Then for any lift $f:\R\rightarrow\R$ of $\varphi|_{\partial D}$ via $\pi$, the limit
\begin{equation}
		\tau(f):=\lim_{n\rightarrow\infty}\frac{f^n(x)-x}{n}\in\R
\end{equation}
exists and is independent of $x$, see for example \cite{Katok_H}, and is called the
translation number of $f$.  Furthermore, the element $[\tau(f)]\in\R/\Z$ in the quotient space,
is even independent of the choice of lift $f$, and is called the rotation number of $\varphi|_{\partial D}$.

\begin{definition}
Let $\varphi:D\rightarrow D$ be an irrational pseudo-rotation.  Then we define the \emph{rotation number}
of $\varphi$ to be the value on the circle
\[
		      \Rot(\varphi):=[\tau(f)]\in\R/\Z
\]
for any lift $f:\R\rightarrow\R$ of the restriction $\varphi:\partial D\rightarrow\partial D$.
\end{definition}

A preferred homotopy $\{\varphi_t\}_{t\in[0,1]}$ from $\varphi_0=\id_{\partial D}$ to
$\varphi_1=\varphi|_{\partial D}$ gives us a preferred lift of $\id_{\partial D}$.  Namely
the terminal map of the unique lift to a homotopy in the universal covering space which begins at
$\id_{\R}$.  In particular, any Hamiltonian generating $\varphi$ as its time-one map, restricts
to a homotopy on the boundary of the disk from $\id_{\partial D}$ to $\varphi|_{\partial D}$ and
thus determines a canonical lift of the latter.  Using this we define:

\begin{definition}
Let $\varphi:D\rightarrow D$ be an irrational pseudo-rotation.  Let $H_t\in C^\infty(D,\R)$ be
a path of Hamiltonians on $(D,\omega_0=dx\wedge dy)$ generating $\varphi$ as its time-one map.
Then we define the \emph{rotation number of $\varphi$ with respect to $H$} to be the real number
\[
		      \Rot(\varphi;H):=\tau(f)\in\R
\]
where $f:\R\rightarrow\R$ is the canonical lift of $\varphi:\partial D\rightarrow\partial D$
determined by $H$.
\end{definition}

If $\varphi$ is an irrational pseudo-rotation, then the unique periodic point is
non-degenerate in the sense that for all $n\in\N$, the linearization
$D\varphi^{(n)}(0)$ does not have eigenvalue $1$.
The proof is a well known application of the Poincar\'e-Birkhoff
fixed point theorem.  See appendix B.

\section{Finite energy foliations} 
For any $H\in C^\infty(\R/\Z\times D,\R)$ with $H_t:=H(t,\cdot)$ constant on
the boundary of $D$ for each $t\in\R/\Z$, the smooth time-dependent vector field 
$X_{H}(t,\cdot):=X_{H_t}$ on $D$ defined by
\[
			    \omega_0(X_{H_t}(z),\cdot)=-dH_t(z)
\]
for all $z=(x,y)\in D$ is tangent to $\partial D$ and therefore generates a $1$-parameter family of
diffeomorphisms $\phi^t:D\rightarrow D$ over $t\in\R$.  Using that the disk is simply connected it is 
well known that one may find an $H$ for any element $\varphi\in\Diff^{\infty}(D,\omega_{0})$ so that 
$\varphi=\phi^1$.  Then $H$ is said to generate $\varphi$.  

From now on let $\varphi:D\rightarrow D$ be a fixed irrational pseudo-rotation.  Unless stated otherwise 
$H\in C^\infty(\R/\Z\times D,\R)$ is a $1$-periodic time-dependent Hamiltonian generating $\varphi$.

\begin{remark}\label{R:normalizing_periodic_orbit}
By precomposing $H$ with a suitable closed loop in $\Diff^\infty(D,\omega_0)$ based at the identity,
we may assume that the unique $1$-periodic
orbit of $X_{H_t}$ corresponding to the fixed point $0\in D$ of $\varphi$ is the constant trajectory
$t\mapsto 0\in D$ for all $t\in\R$.  This is not necessary, but makes the proof of theorem
\ref{T:existence_and_uniqueness_of_folns_for_pseudorotations} slightly easier.
\end{remark}

Define a smooth vector field $R_{H}$ on the solid torus $Z:=\R/\Z\times D$ by
\begin{equation}\label{E:defn_of_R}
		  R_H(\tau,z):=\partial_\tau+X_H(\tau,z)
\end{equation}
for all $\tau\in\R/\Z,z\in D$.
The first return map on $\{0\}\times D$ is canonically identified with the pseudo-rotation
$\varphi$.

For each $n\in\N$ let $Z_n$ be the $3$-manifold-with-boundary $\R/n\Z\times D$, and $R_n$
the vector field on $Z_{n}$ that projects down to $R_{H}$ under the natural projection
$\pi_{n}:Z_{n}\rightarrow Z$.
Clearly the first return map of the flow generated by $R_n$ is the $n$-th iterate
$\varphi^{n}:D\rightarrow D$.  We will refer to the pair $(Z_n,R_n)$ as
\emph{the mapping torus of length-$n$ associated to $H$}.  It will also be useful to
denote by $Z_\infty:=\R\times D$ the universal covering of each $Z_n$.

All the dynamical information on $(Z_n,R_n)$ can be captured by an almost complex structure on the
$4$-manifold $\R\times Z_{n}$ as follows.  For each $n\in\N\cup\{\infty\}$ define $J_{n}$ on $\R\times Z_{n}$
by
\begin{equation}\label{E:defn_of_Jtilde}
 \left\{\begin{aligned}
		&J_n(a,\tau,z)\partial_{\R}=R_n \\
		&J_n(a,\tau,z)|_{T_{z}D}=i
 \end{aligned}\right.
\end{equation}
for all $(a,\tau,z)\in\R\times Z_n$.  Here, $\partial_{\R}$ is the vector field dual to the $\R$-coordinate
on $\R\times Z_{n}$, and $i$ denotes the constant almost complex structure
on the disk coming from the standard integrable complex structure on $\C$
\footnote{For convenience we use this ``non-generic'' almost complex structure, although it is 
not necessary.  The pseudoholomorphic curves we encounter are either orbit cylinders or embedded 
with genus zero, one boundary component, and Fredholm index $2$.  Such curves are automatically regular.}.
In other words $i\partial_x=\partial_y$ and $i\partial_y=-\partial_x$.
Observe that $J_n$ is independent of the $\R$-coordinate on $\R\times Z_{n}$, referred to as \emph{$\R$-invariance}.
This idea of coupling a suitable conservative vector field in an odd-dimensional manifold with the
$\R$-direction in the product $4$-manifold by an almost complex structure is due to Hofer \cite{Hofer_93}.

There is a $1$-parameter family of $2$-tori
\[
		    L_c:=\{c\}\times\partial Z_{n}
\]
for $c\in\R$, that fill the boundary $\R\times\partial Z_{n}$ of the $4$-manifold.  Each $L_{n}$
is totally real with respect to the almost complex structure $J_{n}$, that is
\[
		      TL_c\oplus J_{n}T(L_c)
\]
is the full $4$-dimensional tangent space at each point of $L_c$.  These will form the boundary
conditions for our pseudoholomorphic curves with boundary.

Let us describe the $J_{n}$-holomorphic half infinite cylinders with totally real boundary
conditions that we are interested in.
Let $\R^+=[0,\infty)$ and $\R^-=(-\infty,0]$.  For $n\in\N$ we consider maps
$\u=(a,\tau,z)\in C^\infty(\R^\pm\times\R/n\Z,\R\times Z_n=\R\times\R/n\Z\times D)$ for which there
exists $c\in\R$ such that
\begin{equation}\label{E:hol_curve_equation_1}
\left\{\begin{aligned}
		\partial_s\u(s,t)+J_n(\u(s,t))\partial_t\u(s,t)=0\qquad&\mbox{for all }(s,t)\in\R^\pm\times\R/n\Z \\
		  \u(0,t)\in L_c\qquad&\mbox{for all }t\in\R/n\Z\\
		  \tau(0,\cdot):\R/n\Z\rightarrow\R/n\Z\qquad&\mbox{has degree }+1, 
\end{aligned}\right.
\end{equation}
having so called finite total energy, which we define in a moment.

This setting is a special case of that described in \cite{BEHWZ_SFT_compactness}.
In particular $(\R\times Z_n,J_n)$ is a cylindrical symmetric almost complex manifold, and the almost
complex structure $J_n$ is compatible with the stable Hamiltonian structure
$(\omega_n,\lambda_n)$ on $Z_{n}$ given by 
\begin{equation}\label{E:SHS}
\left\{\begin{aligned}
		 &\omega_n=dx\wedge dy+d\tau\wedge dH\\
		 &\lambda_n=d\tau
\end{aligned}\right.
\end{equation}
in coordinates $(\tau,(x,y))$ on $\R/n\Z\times D$. Recall that this means that 
$\lambda_{n}\wedge\omega_{n}>0$ and $\ker(\omega_{n})\subset\ker(d\lambda_{n})$, see for example 
\cite{BEHWZ_SFT_compactness} or \cite{Cieliebak_Volkov}.  
The compactness theory in \cite{BEHWZ_SFT_compactness}
leads us to consider the following two quantities for a solution
to (\ref{E:hol_curve_equation_1}) which
we will refer to as the $\omega$-energy, the $\lambda$-energy, and the sum of them as the total energy.
In our context the $\lambda$-energy of a solution $\u=(a,\tau,z)\in\R\times\R/n\Z\times D$ to
(\ref{E:hol_curve_equation_1}) is the quantity
\begin{equation}\label{E:lambda_energy}
 E_\lambda(\u):=\sup_{\psi\in\mathcal{C}}\int_{\R^+\times\R/n\Z}\u^{*}\Big(\psi(a)da\wedge d\tau\Big)\in[0,+\infty]
\end{equation}
where $\mathcal{C}$ is the set of smooth functions $\psi:\R\rightarrow[0,\infty)$ for which
$\int_\R\psi(s)ds=1$.  The second energy, that which in the more general context of \cite{BEHWZ_SFT_compactness}
is called the $\omega$-energy, is
\begin{equation}\label{E:omega_energy}
	E_\omega(\u):=\int_{\R^+\times\R/n\Z}\u^*\omega_n \in [0,+\infty].
\end{equation}
In section \ref{S:compactness} we will prove the following.

\begin{lemma}\label{L:nice_form}
Let $n\in\N$.  Suppose $\u\in C^\infty(\R^+\times\R/n\Z,\R\times Z_n)$ is a
solution to (\ref{E:hol_curve_equation_1}) with $E_\lambda(\u)<\infty$.  Then there exists
$(s_0,t_0)\in\R\times\R/n\Z$ so that
\begin{equation}\label{E:nice_form}
		      \u(s+s_0,t+t_0)=(s,t,z(s,t))
\end{equation}
for all $(s,t)\in[-s_0,\infty)\times\R/n\Z$, where $z\in C^\infty([-s_{0},\infty)\times\R/n\Z,D)$ satisfies
the Floer equation
\begin{equation}\label{E:Floer_eqn}
		 \partial_sz(s,t)+i\Big(\partial_tz(s,t)-X_{H}(t,z(s,t))\Big)=0
\end{equation}
for all $(s,t)\in[-s_0,\infty)\times\R/n\Z$.
\end{lemma}

\noindent This is a converse to ``Gromov's trick'' \cite{Gromov}.

It follows that if a solution $\u=(a,\tau,z)$ to (\ref{E:hol_curve_equation_1}) has finite $\lambda$-energy
then $E_\lambda(\u)=n$, and the $\omega$-energy of $\u$ is equal to the Floer energy of $z$;
\[
  E_\omega(\u)=\frac{1}{2}\int_{s=0}^{+\infty}\int_{t=0}^n \big|\partial_sz(s,t)\big|^2+
					    \big|\partial_tz(s,t)-X_{H}(t,z(s,t))\big|^2 ds dt.
\]

\begin{definition}
For a solution $\u=(a,\tau,z)\in C^\infty(\R^\pm\times\R/n\Z,\R\times Z_n)$ to (\ref{E:hol_curve_equation_1})
we refer to the degree of the circle map 
\[
		z(0,\cdot):\R/n\Z\rightarrow\partial D
\]
as the \emph{boundary index} of $\u$.
\end{definition}

\begin{theorem}\label{T:existence_and_uniqueness_of_folns_for_pseudorotations}
Let $H\in C^{\infty}(\R/\Z\times D,\R)$ be a Hamiltonian generating an irrational
pseudo-rotation $\varphi$.  Let $(Z_1,R_1),(Z_2,R_2),\ldots$ be the corresponding sequence
of mapping tori.  For each $n\in\N$ let $\gamma_n:\R/n\Z\rightarrow Z_n$ be the unique $n$-periodic
orbit of $R_n$ for which $\gamma_n(0)\in\{0\}\times D$.  Assume $H$ was chosen so that $\gamma(t)=(t,0)$
for all $t\in\R/\Z$ (see remark \ref{R:normalizing_periodic_orbit}).  Let $\alpha:=\Rot(\varphi;H)\in\R$, 
which is necessarily irrational.  

Then for each $n\in\N$ there exist two foliations $\F_n^+,\F_n^-$ of $\R\times Z_n$ by smoothly
embedded surfaces, with the following properties:
\begin{itemize}
 \item \textbf{Cylinder leaf:} The cylinder $C_n:=\R\times\gamma_n(\R/n\Z)\subset\R\times Z_n$ is a leaf in both
 $\F_n^+$ and $\F_n^-$.
 \item \textbf{Pseudo-holomorphic:} If $F\in\F_n^+$ (resp. $F\in\F_n^-$) is not $C_n$, then $F$ is parameterized by
 a solution $\u$ to (\ref{E:hol_curve_equation_1}) with $J_{n}$ as in 
 (\ref{E:defn_of_Jtilde}), with $E_\lambda(\u)+E_\omega(\u)<\infty$
 and boundary index $\lfloor n\alpha\rfloor$  (resp. $\lceil n\alpha\rceil$).
 \item \textbf{$\R$-invariance:} If $F\in\F_n^+$ (resp. $F\in\F_n^-$) is a leaf and $c\in\R$, the set
 $F+c:=\{(a+c,\tau,z)\,|\,(a,\tau,z)\in F \}$ is also a leaf in $\F_n^+$  (resp. in $\F_n^-$).
 \item \textbf{Uniqueness:} $\F_n^+$ and $\F_n^-$ are uniquely determined by the above properties.
 \item \textbf{Smooth foliation:} $\F_n^+$ and $\F_n^-$ are $C^\infty$-smooth foliations at each point on the
  complement of $C_n$.
\end{itemize}
\end{theorem}

The proof is postponed to section \ref{S:pf_of_fef}.  This a special case of a much more general result
to appear in \cite{Bramham}.

\begin{remark}
For each leaf $F\in\F_n^+$ (resp. $F\in\F_n^-$) that is not the cylinder, any parameterization $\u$
satisfying (\ref{E:hol_curve_equation_1}) has domain $\R^+\times\R/n\Z$
(resp. $\R^-\times\R/n\Z$).  Hence the superscripts in $\F_n^\pm$.  In either case, as the
unique $n$-periodic orbit $\gamma_n$ is non-degenerate the finite energy of $\u$ implies that
the $Z_n$ component $u_\pm(s,\cdot):\R/n\Z\rightarrow Z_n$ converges to $\gamma_n$ uniformly in
$C^\infty(\R/n\Z,Z_n)$ as $s\rightarrow+\infty$ (resp. $s\rightarrow-\infty$).  This can be seen
in two ways.  Either as a consequence of the compactness results in \cite{Hofer_93} applied to $\u$
(as generalized in \cite{HWZ_propI,BEHWZ_SFT_compactness}); or, via lemma \ref{L:nice_form},
the original work of Floer \cite{Floer_unreg_grad_flow} applied to the disk component $z$.
\end{remark}

\begin{remark}
Let $n\in\N$.  The foliations $\F_n^\pm$ of $\R\times Z_n$ can be visualized as follows.
Under the projection map $:\R\times Z_n\rightarrow Z_n$, $\F_n^+$ and $\F_n^-$ project to smooth
foliations of $Z_n\backslash\{\gamma_n\}$ by smoothly embedded surfaces diffeomorphic to $\R/n\Z\times(0,1]$.
The vector field $R_n=\partial_\tau+X_{H_\tau}$ on $Z_n$ is transverse to these leaves coming from
$\F_n^+$ and $\F_n^-$ in opposite directions.  A typical transverse disk slice to either of these
looks something like in figure
\ref{F:transverse_slice1}.
\end{remark}

\begin{figure}[hbt]
\begin{center}
\includegraphics[scale=.3]{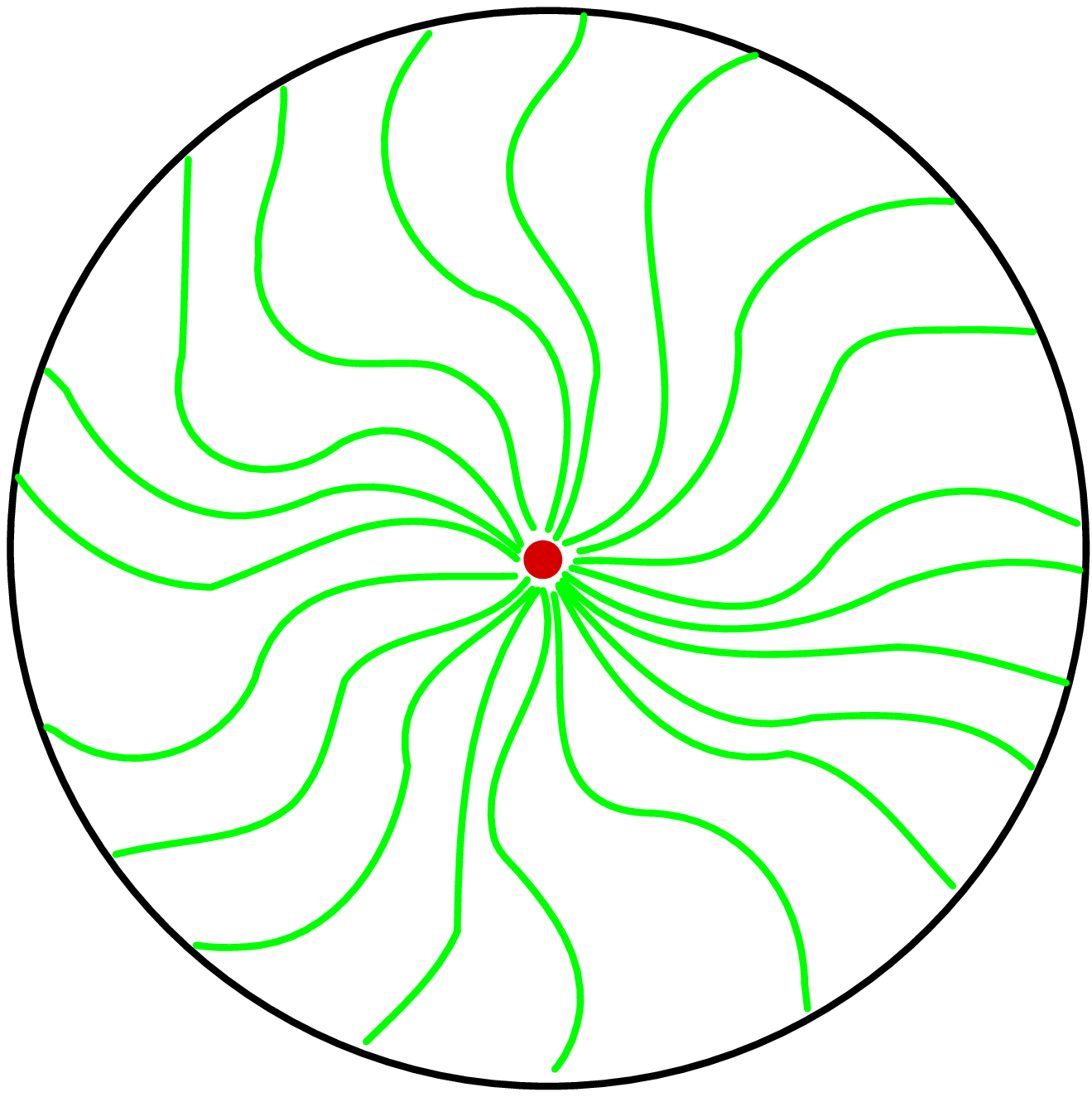}
\end{center}
\caption{}
\label{F:transverse_slice1}
\end{figure}

The following formula will be crucial to our application.  Recall that $\alpha:=\Rot(\varphi;H)\in\R$.  
In particular $\alpha$ is irrational.

\begin{lemma}\label{L:formula_1}
Let $n\in\N$.  For any half cylinder leaf $F\in\F_n^+$,
\[
		      E_\omega(F)=\{n\alpha\}\pi
\]
where $\{\cdot\}$ denotes the fractional part of a real number.
\end{lemma}
This is proven in section \ref{S:omega_energy}.

\section{Proof of theorem \ref{T:main_result_intro}}\label{S:pf_main_result}
We use the finite energy foliations of theorem
\ref{T:existence_and_uniqueness_of_folns_for_pseudorotations} to define new disk maps.

\begin{definition}\label{D:approx_map}
For each $n\in\N$ define $\varphi_{n}:D\rightarrow D$
as follows.  For $\xi\in D$ there is a unique leaf $F\in\F_{n}^{+}$ containing
$(0,0,\xi)\in\R\times\R/n\Z\times D$.  Define $\varphi_n(\xi)=\xi'$
where $\xi'\in D$ is unique such that $(0,0,\xi')\in F$.
\end{definition}

\begin{remark}
$\varphi_n$ is well defined: by lemma \ref{L:nice_form} if a leaf $F\in\F_{n}^{+}$ intersects the hypersurface
\[
	  \Zhat_n:=\{0\}\times Z_{n}
\]
then it does so transversally, and for each $\tau\in\R/n\Z$ it will intersect the disk
slice $\{\tau\}\times D\subset\Zhat_{n}$ in a unique point.
\end{remark}

\begin{remark}
We could as easily define maps in terms of the foliations $\F_n^-$.
\end{remark}

\begin{remark}
The inverse map $\varphi_{n}^{-1}$ exists and can be defined similarly in terms of $\F_n^+$.
\end{remark}

\begin{lemma}
Each map $\varphi_{n}:D\rightarrow D$ is $n$-periodic, that is
$(\varphi_{n})^{n}=\id_{D}$.
\end{lemma}
\begin{proof}
We could define $n$ many disk maps using $\F_{n}^{+}$; say $\varphi_{n,i}$ for
$i=0,1,\ldots,n-1$, by requiring that $\varphi_{n,i}$ takes the point $\xi\in D$
to $\xi'\in D$ if $(0,i,\xi)$ and $(0,i+1,\xi')$ lie on the same leaf in $\F_{n}^{+}$.
Since each leaf in the foliation closes up after going once around in the $\R/n\Z$ direction,
it follows that the composition $\varphi_{n,n-1}\circ,\ldots,\circ\varphi_{n,0}$
is the identity map.

We now exploit a symmetry in $\F_{n}^{+}$ to see that each of the maps $\varphi_{n,i}$ is equal
to $\varphi_{n}$.  Consider the $\Z_{n}$ action on $\R\times Z_{n}$ generated by the
deck transformation
\begin{align*}
					\T:\R\times Z_{n}\rightarrow\R\times Z_{n}\\
						\T(a,\tau,z)=(a,\tau-1,z)
\end{align*}
which preserves the almost complex structure; $\T^{*}J_{n}=J_{n}$.
From the uniqueness part of theorem \ref{T:existence_and_uniqueness_of_folns_for_pseudorotations}
we conclude that the foliation $\T(\F_{n}^{+}):=\{\T(F)\,|\,F\in\F_{n}^{+}\}$ is equal to $\F_{n}^{+}$.
Hence $\varphi_{n,i}=\varphi_{n}$ for each $i$.
\end{proof}

Note that $\varphi_{n}(0)=0$ as the cylinder $C_{n}$ passes through the center of
the disk slices $\{0\}\times D$ and $\{1\}\times D$ in $\Zhat_{n}=\{0\}\times Z_{n}$.

\begin{lemma}\label{T:smoothness_of_foliation_maps}
$\varphi_{n}:D\rightarrow D$ is $C^{\infty}$-smooth on $D\backslash\{0\}$.
\end{lemma}
\begin{proof}
Recall that by theorem \ref{T:existence_and_uniqueness_of_folns_for_pseudorotations}
$\F_{n}^{+}$ is a $C^{\infty}$-smooth foliation on the complement of the
leaf $C_{n}$.  From lemma \ref{L:nice_form} the leaves of $\F_{n}^{+}$ are
uniformly transverse to the hypersurface $\Zhat_{n}=\{0\}\times Z_{n}$, and within this
to all the disk slices, in particular to $\{0\}\times D$ and $\{1\}\times D$.  Now consider any
point $\xi\in D\backslash\{0\}$.  Then $\xi':=\varphi_{n}(\xi)\in D\backslash\{0\}$.
Therefore there exist local smooth foliation charts about $(0,0,\xi)$ and $(0,1,\xi')$.  In
these charts the map $\varphi_{n}$ from its definition has a smooth expression.
\end{proof}

\begin{lemma}
$\varphi_n:D\rightarrow D$ is continous at the origin.
\end{lemma}
\begin{proof}
Let $\xi_j\in D\backslash\{0\}$ be a sequence of points converging to $0\in D$, and
$p_j:=(0,0,\xi_j)\in\R\times Z_n$.  By lemma \ref{L:compactness_of_each_foliation} there
exists a sequence of parameterizations $\u_j$ of the unique leaf $F_j\in\F_n^+$ containing $p_j$
that converges in the $C^\infty_{\loc}$-topology to $C_n$.  By lemma \ref{L:nice_form}
we may choose each $\u_j$ to take the form $\u_j:[S_j,\infty)\times\R/n\Z\rightarrow\R\times Z_n$
for some $S_j\in\R$, with
\[
			  \u_j(s,t)=(s,t,z_j(s,t))
\]
for some sequence $z_j$.  It follows that the sequence of points
\[
		    F_j\cap\{(0,1)\}\times D=(0,1,z_j(0,1))
\]
converges to $(0,1,0)$.  Which means that $\varphi_n(\xi_j)=z_j(0,1)\rightarrow 0$ as
$j\rightarrow\infty$.
\end{proof}

Fix a subsequence $\{n_j\}_{j\in\N}$ for which the sequence of fractional parts
\begin{align}\label{E:subsequence}
		    \lim_{j\rightarrow\infty}\{n_j\alpha\}=0.
\end{align}
From lemma (\ref{L:formula_1}) this implies that the $\omega$-energies of the
leaves in the foliations $\F_{n_j}^+$ tends to zero, as $j\rightarrow\infty$, uniformly over all leaves.

For maps $f,g:D\rightarrow D$, define $d_{C^0}(f,g)$ using the linear structure and Euclidean norm on
$\R^2$ by
\[
      d_{C^0}(f,g):=\sup_{\xi\in D}|f(\xi)-g(\xi)|.
\]
\begin{proposition}\label{T:uniform_convergence}
The subsequence $\varphi_{n_j}$ converges to the
pseudo-rotation $\varphi$ in the following sense:
\begin{align*}
	d_{C^0}(\varphi_{n_j},\varphi)+d_{C^0}(\varphi_{n_j}^{-1},\varphi^{-1})\rightarrow0
\end{align*}
as $j\rightarrow\infty$.
\end{proposition}
\begin{proof}
We show that $d_{C^0}(\varphi_{n_j},\varphi)\rightarrow0$ as $j\rightarrow\infty$, as the same
argument will work for the inverses.

Arguing indirectly, there exists a sequence of points $\xi_{j}\in D$ and
$\delta>0$ such that $|\varphi_{n_{j}}(\xi_{j})-\varphi(\xi_{j})|\geq\delta$
for all $j\in\N$, where $|\cdot|$ is the Euclidean norm on $\R^2$. Restricting to a
subsequence we may assume that $\xi_{j}\rightarrow \xi$ for some $\xi\in D$, and
$\delta\leq|\varphi_{n_{j}}(\xi_{j})-\varphi(\xi_{j})|\leq|\varphi_{n_{j}}(\xi_{j})-\varphi(\xi)|+
|\varphi(\xi_{j})-\varphi(\xi)|$ for all $j\in\N$.
Therefore as $\varphi$ is continous,
\begin{equation}\label{E:487}
		\half\delta\leq|\varphi_{n_{j}}(\xi_{j})-\varphi(\xi)|
\end{equation}
for all $j$ sufficiently large.

For each $j\in\N$ let $F_{j}\in \F_{n_{j}}^+$ be the unique leaf containing the
point $(0,0,\xi_{j})\in\R\times\R/n_{j}\Z\times D$.  Let us assume that each $F_j$ is a half cylinder,
otherwise the argument is even easier.  There exists a solution $\u_j$ to (\ref{E:hol_curve_equation_1})
parameterizing $F_j$.  After a holomorphic reparameterization we may assume that
\begin{align*}
			\u_{j}:[S_{j},\infty)\times\R/n_{j}\Z &\rightarrow\R\times Z_{n_{j}} \\
					\u_{j}(0,0)= & (0,0,\xi_{j})
\end{align*}
for some $S_{j}\leq0$.  For each $j$, $E_{\lambda}(\u_{j})<\infty$, so by lemma
\ref{L:nice_form}, $\u_j$ takes the form
\begin{equation}\label{E:form_of_converging_curves}
 			\u_j(s,t)=(s,t,z_j(s,t))
\end{equation}
some $z_j:[S_{j},\infty)\times\R/n_{j}\Z\rightarrow D$.
Moreover, the sequence $\{\u_j\}_{j\in\N}$ satisfies all the criterion for
the compactness result theorem \ref{T:compact_sequences}.  In particular
\begin{align*}
		\lim_{j\rightarrow\infty}E_{\omega}(\u_j)=0
\end{align*}
due to our choice of subsequence satisfying (\ref{E:subsequence}).  We conclude that
the sequence $\u_j$ converges in the following sense: for each $j$ let
$\ubar_j:[S_j,\infty)\times\R\rightarrow\R\times Z_{\infty}$ be the unique lift of $\u_j$ to
the universal covering, satisfying
\begin{equation}\label{E:origin_to_basepoints}
 			\ubar_j(0,0)=(0,0,\xi_{j}).
\end{equation}
After restricting to a further subseqence we can assume that $\ubar_j\rightarrow\ubar_\infty$
in the $C^\infty_{\loc}(\R^2,\R\times Z_\infty)$ topology, that is, uniformly on compact sets,
where $\ubar_\infty$ takes the form
\[
		      \ubar_\infty(s,t)=(s,t,\gamma(t))
\]
for some $\gamma\in C^\infty(\R,D)$ solving $\gammadot(t)=X_{H}(t,\gamma(t))$ for all $t\in\R$.
From (\ref{E:origin_to_basepoints}) we have $\ubar_\infty(0,0)=(0,0,\xi)$.   We conclude that
for all $j$,
\begin{align*}
 	  (0,1,\varphi_{n_j}(\xi_j))&=\ubar_j(0,1).
\intertext{The right hand side converges to }
		  \ubar_\infty(0,1)&=(0,1,\gamma(1))\\
			    &=(0,1,\varphi(\gamma(0)))\\
			    &=(0,1,\varphi(\xi)).
\end{align*}
This contradicts (\ref{E:487}) and we are done.
\end{proof}

Combining the results of this section we have proven the following statement which is almost
theorem \ref{T:main_result_intro}.

\begin{theorem}\label{T:main_result}
Suppose $\varphi\in\Diff^\infty(D,\omega_0)$ fixes the origin and has no other periodic points.
Then there exists a sequence of maps $\varphi_j\in\Homeo_+(D)\cap\Diff^\infty(D\backslash\{0\})$
over $j\in\N$, with the following properties.  For each $j\in\N$, $\varphi_j(0)=0$,
there exists $n_j\in\N$ such that $\varphi_j^{n_j}=\id_D$, and
$d_{C^0}(\varphi_j,\varphi)+d_{C^0}(\varphi_j^{-1},\varphi^{-1})\rightarrow 0$ as $j\rightarrow\infty$.
\end{theorem}

There are presumably nicer ways to go from this conclusion to the final statement.  For example using 
changes of coordinates from the pseudoholomorphic curves themselves.  This will presumably follow from 
a more serious analysis of the asymptotic properties of the curves.  

\begin{proof}[Proof of theorem \ref{T:main_result_intro}]
It is a classical result \cite{Brouwer,Eilenberg,Kere} that if $f\in\Homeo^+(D)$ satisfies $f^n=\id_D$
for some $n\in\N$, then there exists $g\in\Homeo^+(D)$ and $q\in\{0,1,\ldots,n-1\}$ so that
\[
		      g\circ f\circ g^{-1}=R_{2\pi q/n}.
\]
If $q\neq0$ then $g$ must fix the origin, and if $q=0$ then $f=\id_D$ anyway, so we may assume
$g$ fixes the origin.  
Applying this to each $\varphi_{j}$ we find $g_j\in\Homeo^+(D)$, fixing the origin, and
$p_j\in\Z$ such that
\[
		  \varphi_{j}=g_j^{-1}\circ R_{2\pi p_j/n_j}\circ g_j.
\]
Now we replace $g_j$ by a $C^0$-close smooth approximation.  More precisely, let $\ghat_j$ be a sequence
in $\Diff^\infty(D)$, each fixing the origin, with
$d_{C^0}(\ghat_j,g_j)+d_{C^0}(\ghat_j^{-1},g_j^{-1})\rightarrow 0$ as $j\rightarrow\infty$.
Then the maps $\varphihat_j:=\ghat_j^{-1}\circ R_{2\pi p_j/n_j}\circ\ghat_j$ are
$C^\infty$-diffeomorphisms which converge in the $C^0$-sense to the irrational pseudo-rotation $\varphi$.
The maps $\varphihat_j$ satisfy the conditions of theorem \ref{T:main_result_intro}.
\end{proof}

\section{Calculation of the $\omega$-energy}\label{S:omega_energy}
Choose a $1$-form $\lambda_0$ on the disk so that $d\lambda_0=\omega_0=dx\wedge dy$.  For each $n\in\N$ define
the action functional $\Act_n:C^\infty(\R/n\Z,Z_n)\rightarrow\R$ (associated to $\lambda_0$) by
\begin{equation}\label{E:defn_of_Action_n}
\begin{aligned}
			\Act_n(\sigma):=\int_{\R/n\Z}\sigma^*\lambda_0 - \int_0^nH(\sigma(t))dt.
\end{aligned}
\end{equation}
We may rewrite this as
\begin{align}\label{E:defn_of_Action_n_1}
		\Act_n(\sigma):=\int_{\R/n\Z}\sigma^*\eta_n
\end{align}
where $\eta_n:=\lambda_0 - H d\tau$ is a primitive of $\omega_n$ the $2$-form used to define the
$\omega$-energy.  Note that $\eta_n$ restricts to a closed
$1$-form on $\partial Z_n$ since $R_n$ is tangent to $\partial Z_n$ and
$d\eta_n(R_n,\cdot)=\omega_n(R_n,\cdot)=0$.
Hence $\Act_n$ restricted to $C^\infty(\R/n\Z,\partial Z_n)$
descends to a map on homology. 

\begin{lemma}\label{L:action_inequalities_on_gamma_n}
For each $n\in\N$
\begin{equation}\label{E:action_inequalities_on_gamma_n}
 \Act_n(1_{\R/n\Z}) + \lfloor n\alpha\rfloor\Act_n(1_{\partial D})\leq \Act_n(\gamma_n)
			      \leq \Act_n(1_{\R/n\Z}) + \lceil n\alpha\rceil\Act_n(1_{\partial D}).
\end{equation}
\end{lemma}
\begin{proof}
Let $\u_\pm=(a_\pm,u_\pm):\R^\pm\times\R/n\Z\rightarrow\R\times Z_n$ be parameterizations
of leaves $F^\pm\in\F_n^\pm$ respectively which satisfy (\ref{E:hol_curve_equation_1}).
In either case $u_\pm(s,\cdot)$ converges uniformly in $C^\infty(\R/n\Z,Z_n)$ to a parameterization
$\gamma_n(\textup{const}_\pm+\cdot)$ as $s\rightarrow\pm\infty$ respectively.  Applying Stokes theorem,
\begin{align*}
	E_\omega(\u_+)=\int_{\R^+\times\R/n\Z}u_+^*\omega_n
		      =\Act_n(\gamma_n) - \Act_n(u_+(0,\cdot)),
\end{align*}
and
\begin{align*}
	E_\omega(\u_-)=\int_{\R^-\times\R/n\Z}u_-^*\omega_n
		      =\Act_n(u_-(0,\cdot))-\Act_n(\gamma_n).
\end{align*}
Therefore, as the energies are non-negative,
\[
			 \Act_n(u_+(0,\cdot))\leq \Act_n(\gamma_n)\leq \Act_n(u_-(0,\cdot)).
\]
We observed that the action $\Act_n$ of a closed loop in $\partial Z_n$ depends only on its
homology class.  From theorem \ref{T:existence_and_uniqueness_of_folns_for_pseudorotations}
$u_+(0,\cdot):\R/n\Z\rightarrow\R/n\Z\times\partial D$ and
$u_-(0,\cdot):\R/n\Z\rightarrow\R/n\Z\times\partial D$
are homologous to $t\mapsto(t,e^{2\pi i(\lfloor n\alpha\rfloor/n)t})$ and
$t\mapsto(t,e^{2\pi i(\lceil n\alpha\rceil)t})$ respectively.  Therefore we get the inqualities
in (\ref{E:action_inequalities_on_gamma_n}).
\end{proof}

\begin{corollary}\label{L:action_of_gamma_1}
The unique $1$-periodic orbit $\gamma_1:\R/\Z\rightarrow Z_1$ has action
\begin{equation}\label{E:action_of_gamma_1}
	  \Act_1(\gamma_1)=\Act_1(1_{\R/\Z}) + \alpha\Act_1(1_{\partial D})
\end{equation}
where $1_{\R/\Z}$ and $1_{\partial D}$ are the closed loops in $\partial Z=\R/\Z\times\partial D$
given by $t\mapsto(t,\textup{pt})$ and $t\mapsto(\textup{pt},t)$ respectively.
\end{corollary}
\begin{proof}
From the definition of $\Act_n$,
\begin{align*}
	      \Act_n(\gamma_n)&=n\Act_1(\gamma_1)\\
	      \Act_n(1_{\R/n\Z})&=n\Act_1(1_{\R/\Z})\\
	      \Act_n(1_{\partial D})&=\Act_1(1_{\partial D})
\end{align*}
for all $n\in\N$.  Substituting these into the inequalities in lemma
(\ref{L:action_inequalities_on_gamma_n}), and dividing through by $n$ and letting
$n\rightarrow+\infty$ gives (\ref{E:action_of_gamma_1}).
\end{proof}

\begin{lemma}\label{P:omega_energy_of_half_cylinders_in_F_n}
Let $n\in\N$.  Every leaf $F\in\F_n^+$ with boundary has
$\omega$-energy
\[
		  E_{\omega}(F)=\{n\alpha\}\pi
\]
where $\{\cdot\}$ applied to any real number denotes its fractional part.
\end{lemma}
\begin{proof}
By Stokes theorem as in the last lemma,
\[
		    E_{\omega}(F)=\Act_n(\gamma_n) - \Act_n(u(0,\cdot))
\]
where $\u=(a,u):\R^+\times\R/n\Z\rightarrow\R\times Z_n$ is a parameterization of $F$.
Using corollary \ref{L:action_of_gamma_1} and that $u(0,\cdot):\R/n\Z\rightarrow\R/n\Z\times\partial D$
is homologous to $t\mapsto(t,e^{2\pi i(\lfloor n\alpha\rfloor/n)t})$, this becomes
\begin{align*}
  E_{\omega}(F)&=n\Big(\Act_1(1_{\R/\Z}) + \alpha\Act_1(1_{\partial D})\Big)-\\
			&\hspace{100pt}\Big(\Act_n(1_{\R/n\Z}) + \lfloor n\alpha\rfloor\Act_n(1_{\partial D})\Big)\\
		&=\Big(n\Act_1(1_{\R/\Z}) + n\alpha\Act_1(1_{\partial D})\Big)-\\
			&\hspace{100pt}\Big(n\Act_1(1_{\R/\Z}) + \lfloor n\alpha\rfloor\Act_1(1_{\partial D})\Big)\\
			&=(n\alpha-\lfloor n\alpha\rfloor)\Act_1(1_{\partial D})\\
			&=\{n\alpha\}\cdot\int_{D}dx\wedge dy.
\end{align*}
\end{proof}

\section{Compactness}\label{S:compactness}

Norms, such as $\|\nabla\u_n\|_{L^\infty}$, are implicitely with respect to the Riemannian metric
$dx^2+dy^2+d\tau^2+da^2$ on $\R\times Z_n$, where $(x,y)$ are the standard Euclidean coordinates
on the disk, $\tau$ is the ``coordinate'' on $\R/n\Z$ and $a$ is the $\R$-coordinate.  This
metric is $J_n$-invariant.

\subsection{Compactness when $n$ is uniformly bounded.}

We wish to prove lemma \ref{L:nice_form}.  This will follow from the next two lemmas.

\begin{lemma}[The Floer equation from the Cauchy-Riemann equations]\label{L:form_of_curves_with_bounded_grad}
Let $\u=(a,\tau,z):\R^+\times\R/n\Z\rightarrow\R\times\R/n\Z\times D$ be a solution
to (\ref{E:hol_curve_equation_1}) for which $\|\nabla\u\|_{L^\infty}<\infty$.
Then there exist $(a_0,\tau_0)\in\R\times\R/n\Z$ such that
\begin{equation}\label{E:a_and_tau_are_linear}
 \left\{\begin{aligned}
         a(s,t)&=s + a_0 \\
         \tau(s,t)&=t+\tau_0
        \end{aligned}
	\right.
\end{equation}
for all $(s,t)\in\R^+\times\R/n\Z$, and moreover $z:\R^+\times\R/n\Z\rightarrow D$ satisfies the following
Floer equation:
\begin{equation}\label{E:disk_component_satisfies_floer_equation}
	  \partial_sz(s,t)+i\Big(\partial_tz(s,t)-X_H(t+\tau_0,z(s,t))\Big)=0
\end{equation}
for all $(s,t)\in\R^+\times\R/n\Z$.
\end{lemma}
\begin{proof}
Writing out in coordinates what it means for $\u$ to satisfy (\ref{E:hol_curve_equation_1})
gives us
 \begin{equation}\label{E:1007777}
 (a_s-\tau_t)\partial_a+(a_t+\tau_s)\partial_\tau+
    \Big(a_tX_H(\tau,z)+z_s+i\big(z_t-\tau_tX_H(\tau,z)\big)\Big)=0.
\end{equation}
The boundary condition on $\u$ in (\ref{E:hol_curve_equation_1}) implies $a_t(0,t)=0$
for all $t\in\R/n\Z$.  From (\ref{E:1007777}),
\begin{equation}\label{E:eqns_for_a_and_tau}
 \begin{aligned}
					a_{s}(s,t)&=\tau_{t}(s,t)\\
					a_{t}(s,t)&=-\tau_{s}(s,t)
 \end{aligned}
\end{equation}
for all $(s,t)\in\R^{+}\times\R/n\Z$.  In particular both functions $a,\tau$ lift to
harmonic functions on the upper half plane with gradient bounded in $L^{\infty}$.
The boundary conditions on $a$ allow a smooth extension by reflection to the whole plane,
still with gradient in $L^{\infty}$, and therefore by Liouville the partial derivatives of $a$
are constant.  So there exists $b,c,a_0\in\R$ so that
$a(s,t)=cs+bt+a_0$ for all $(s,t)\in\R^{+}\times\R/n\Z$.  Putting this into (\ref{E:eqns_for_a_and_tau})
there exists $\tau_{0}\in\R$ so that $\tau(s,t)=ct-bs+\tau_0$ for all $(s,t)\in\R^{+}\times\R/n\Z$.
The $n$-periodicity of $a$ in the $t$ variable implies $b=0$.
Also $\tau(s,t+n)=\tau(s,t)+n$ for all $(s,t)\in\R^{+}\times\R$ because
we are assuming that $\deg(\u)=1$.  Therefore $c=1$.  This proves (\ref{E:a_and_tau_are_linear}).

From (\ref{E:1007777}) we also have $a_tX_H(\tau,z)+z_s+i\big(z_t-\tau_tX_H(\tau,z)\big)=0$.
But we have shown that $\tau_t\equiv1$ and $a_t\equiv0$.  Substituting these in we obtain
\[
				z_s+i\big(z_t-X_H(t+\tau_0,z)\big)=0
\]
as required.
\end{proof}

The next statement says that we can use the above relation between the Cauchy-Riemann and
Floer equations if (and only if) the $\lambda$-energy is finite.

\begin{lemma}\label{L:finite_lambda_energy_implies_grad_a_bounded}
Let $\u=(a,\tau,z)$ be a solution to (\ref{E:hol_curve_equation_1}).  Then
$E_\lambda(\u)<\infty$ implies $\|\nabla a\|_{L^\infty}<\infty$
(equivalently $\|\nabla\tau\|_{L^\infty}<\infty$).
\end{lemma}
\begin{proof}
The equations (\ref{E:a_and_tau_are_linear}) in the last lemma did not require the gradient bounds,
and so the map $f:\R^+\times\R/n\Z\rightarrow\R\times\R/n\Z$ given by $f(s,t):=(a(s,t),\tau(s,t))$
in terms of the $a$ and $\tau$ components of $\u$, is holomorphic.  Also
$f(0,t)\in\{c\}\times\R/n\Z$ for all $t\in\R/n\Z$.

Arguing indirectly suppose that the gradient of $a$ is unbounded.  Then the gradient must
blow up along a sequence of points that leaves every compact subset of the domain, in particular
does not converge to the boundary.  Therefore standard rescaling arguments applied to $f$ yield a
holomorphic plane $g:\C\rightarrow\R\times\R/n\Z$ with the following properties:
\begin{align*}
			    |\nabla g(0)|&=1 \\
			     |\nabla g(\xi)|&\leq2\qquad\mbox{for all }\xi\in\C\\
			     E_{\lambda}(g)&<\infty.
\end{align*}
Indeed, it is easily checked that $E_{\lambda}(g)\leq E_{\lambda}(\u)$.
The first two properties imply that $g$ has constant, non-zero, gradient from Liouville's theorem.
But this implies the contradiction $E_{\lambda}(g)=+\infty$.
\end{proof}

In our proof that the maps $\varphi_{n}$ are continuous we used the following.

\begin{lemma}\label{L:compactness_of_each_foliation}
For each fixed $n\in\N$ $\F_n^+$ is compact in the following sense.  Suppose that $F_j\in\F_n^+$ is a
sequence of leaves over $j\in\N$, and $p_j\in F_j$ is a sequence of points.  Suppose that
$p_j\rightarrow p$ some $p\in\R\times Z_n$.  Then there exists a sequence of parameterizations
of $F_{j_k}$ which converge in a $C^\infty_{\loc}$-sense to a parameterization of
the unique leaf in $\F_n^+$ containing $p$.
\end{lemma}
\begin{proof}
This is a standard property of finite energy foliations from positivity of intersections, used
many times in \cite{HWZ_fef} as the $\lambda$-energy and $\omega$-energy are uniformly bounded in $j$.
\end{proof}

\subsection{Compactness as $n\rightarrow\infty$.}
In our proof of convergence of the disk maps $\varphi_n$ in proposition \ref{T:uniform_convergence},
we used a compactness statement for a sequence of $J_{n}$-holomorphic
maps $\u_n:[S_n,\infty)\times\R/n\Z\rightarrow\R\times Z_n$ for which
\begin{align*}
		  E_\lambda(\u_n)&=n\rightarrow+\infty \\
		  E_{\omega}(\u_n)&=\{n\alpha\}\pi
\end{align*}
for some irrational real number $\alpha$.  Hence the total energy
$E(\u_n)= E_\lambda(\u_n)+ E_{\omega}(\u_n)$ diverges to $+\infty$.  In general, for a sequence
of maps $\{\u_n\}$ for which the total energy is unbounded one cannot expect uniform bounds on the gradient
in $L^\infty$.  However if the $\lambda$-energy grows at most linearly with $n$, and the
$\omega$-energy is bounded then indeed uniform bounds on $\|\nabla\u_n\|_{L^\infty}$ can be
achieved.  (Actually much weaker assumptions suffice, but we will not need to explore these here.)
Our arguments will be further simplified since we restricted to a subsequence for which the
$\omega$-energy of the sequence decays to zero.

Consider a sequence  $\{\u_n\}_{n\in\N}$ of smooth $J_n$-holomorphic maps, and numbers $c_n,S_n\in\R$
with $S_{n}\leq0$, satisfying for each $n$,
\begin{equation}\label{E:hol_curve_equation_with_boundary_4}
\left\{\begin{aligned}
		    &\u_n=(a_n,\tau_n,z_n):[S_n,\infty)\times\R/n\Z\rightarrow\R\times Z_{n}, \\
		    &\partial_s\u_n(s,t)+J_n(\u_n(s,t))\partial_t\u_n(s,t)=0 \\
		    &\u_n(S_n,t)\in L_{c_n}\\
		    &\tau_n:(S_n,\cdot):\R/n\Z\rightarrow\R/n\Z\qquad\mbox{has degree }1
\end{aligned}\right.
\end{equation}
for all $(s,t)\in[S_n,\infty)\times\R/n\Z$.

\begin{proposition}\label{T:uniform_grad_bounds}
Suppose that $E_{\lambda}(\u_n)<\infty$ for each $n\in\N$, and
$\lim_{n\rightarrow\infty}E_{\omega}(\u_n)=0$.  Then there exists
$C\in(0,\infty)$ such that
\begin{equation*}
			\|\nabla\u_n\|_{L^\infty([S_n,\infty)\times\R/n\Z)}\leq C
\end{equation*}
for all $n\in\N$.
\end{proposition}

Note that we do not assume uniform bounds on the $\lambda$-energy.

\begin{proof}
Since $E_{\lambda}(\u_n)<\infty$ for each $n$, lemma \ref{L:finite_lambda_energy_implies_grad_a_bounded}
implies $\|\nabla a_n\|_{L^\infty}<\infty$ (for each $n$).  Therefore, since also
each $\tau_{n}$ has degree $1$, lemma \ref{L:form_of_curves_with_bounded_grad} applies so
\begin{equation*}\label{E:a_and_tau_are_linear_1}
 \left\{\begin{aligned}
         a_n(s,t)&=s + a_n \\
         \tau_n(s,t)&=t+\tau_n
        \end{aligned}
	\right.
\end{equation*}
for all $(s,t)\in[S_n,\infty)\times\R/n\Z$, some constants $(a_n,\tau_n)\in[S_n,\infty)\times\R/n\Z$.
Thus
\begin{equation*}\label{E:a_and_tau_are_linear}
 		\|\nabla a_n\|_{L^\infty}\leq 1 \qquad\mbox{and}\qquad\|\nabla\tau_n\|_{L^\infty}\leq1
\end{equation*}
for all $n\in\N$.

It therefore remains to show that the gradients of the $z_{n}$ are uniformly bounded.  Arguing indirectly
we find a sequence $\xi_n\in[S_n,\infty)\times\R/n\Z$ for which
$|\nabla\u_n(\xi_n)|\geq|\nabla z_n(\xi_n)|\rightarrow\infty$ as $n\rightarrow\infty$.
A standard rescaling argument produces a $J_{\infty}$-holomorphic plane or half plane in
$\R\times Z_{\infty}$.
That is, a map $\v:\C\rightarrow\R\times Z_{\infty}$, or 
$\v:\H\rightarrow\R\times Z_{\infty}$ with totally real boundary conditions
$\v(\partial\H)\subset\{c\}\times\partial Z_{\infty}$ for some $c\in\R$.
In either case, as a result of the rescaling process, $\v=(a,\tau,z)$ has the
following properties:
\begin{align*}
					\nabla a\equiv&0\\
					\nabla\tau\equiv&0\\
			  |\nabla\v(0)|&>0, \\
			  E_{\omega}(\v)=&0.
\end{align*}
The first two properties are because of the uniform bounds $\|\nabla a_{j}\|_{L^\infty}\leq 1$
and $\|\nabla\tau_{j}\|_{L^\infty}\leq 1$ respectively even before rescaling, so that rescaling
``kills'' these terms in the limit.  The vanishing $\omega$-energy is because
$E_{\omega}(\v)\leq\lim_{n\rightarrow\infty}E_{\omega}(\u_n)=0$ by Fatou's lemma.
Thus there exist constants $a_{0},\tau_{0}\in\R$ such that
\begin{align*}
	      \v(s,t)=(a_{0},\tau_{0},z(s,t))\in\R\times\R\times D
\end{align*}
for all $(s,t)\in\C$ (resp. all $(s,t)\in\H$).  That $\v$ is $J_\infty$-holomorphic translates
into $z:\C\rightarrow D$ or $z:\H\rightarrow D$ satisfying the equation
$a_tX_H(\tau,z)+z_s+i\big(z_t-\tau_tX_H(\tau,z)\big)=0$, see (\ref{E:1007777}).
So $a$ and $\tau$ constant implies $z_s+iz_t=0$.  (We could alternatively have just rescaled
the sequence $\{z_{n}\}$ as in Floer theory, to get the same conclusion.) Therefore,
\begin{align*}
 	0= E_{\omega}(\v)&=\int\v^*\omega=\int \frac{1}{2}\left(|z_s|^2+|z_t|^2\right)dsdt
\end{align*}
and so $z$ is also constant.  Thus we have shown that $\v$ is constant, contradicting
$|\nabla\v(0)|>0$.
\end{proof}

Now standard arguments can convert these uniform bounds on the gradient in $C^0$ to uniform
$C^k$-bounds on the gradient for all $k\in\N$.
The key result is the following local statement which is proven using the $W^{k,p}$-elliptic
estimates for the linear Cauchy-Riemann operator.

For each $r\geq0$ let $D_r:=\{(x,y)\in\R^2|x^2+y^2\leq r\}$.

\begin{theorem}[Hofer\cite{Hofer_93}]\label{T:grad_bds_implies_Cinfty_bds}
Let $J$ be a fixed, smooth, almost complex structure on $\R^{2d}$, $d\in\N$.
Let $C\in(0,\infty)$.
Consider the set of maps $\mathcal{B}(J,C)\subset C^{\infty}(D_1,\R^{2d})$, consisting of
all $f$ satisfying,
\begin{align*}
						\partial_{s}f+J(f)\partial_{t}f&=0\\
								|f(0)|&\leq 2\\
								\|\nabla f\|_{C^{0}(D_1)}&<C.
\end{align*}
Then for all $r\in(0,1)$ there exists a sequence $c_{k}\in(0,\infty)$ over $k\in\N$,
such that for all $f\in\mathcal{B}(J,C)$,
\begin{equation*}
						\|\nabla f\|_{C^{k}(D_{r})}<c_{k}
\end{equation*}
for all $k\in\N$.
\end{theorem}

Replacing $D_1,D_r$ for half disks $D^+_1,D^+_r$, where $D^+_r:=\{(x,y)\in\R^2|x^2+y^2\leq r,\ y\geq0\}$,
the same statement holds for maps that take the boundary points
$[-1,1]\times\{0\}$ into a smooth path of $J$-totally real subspaces in $\R^{2d}$.

\begin{corollary}\label{C:uniform_global_bds_in_Cinfty}
Suppose $c_0\in(0,\infty)$.  Then there exists a sequence $c_{k}\in(0,\infty)$ over $k\in\N$,
with the following property.  If $\{\u_n\}_{n\in\N}$ is a sequence of solutions to
(\ref{E:hol_curve_equation_with_boundary_4}) such that
$\sup_{n\in\N}\|\nabla\u_n\|_{L^\infty([S_n,\infty)\times\R/n\Z)}\leq c_0$, then
\begin{equation}\label{E:368}
			\|\nabla\u_n\|_{C^{k}([S_n,\infty)\times\R/n\Z)}<c_{k}
\end{equation}
for all $k\in\N$ and $n\in\N$.
\end{corollary}
\begin{proof}
This follows easily from the local result, theorem \ref{T:grad_bds_implies_Cinfty_bds}, using
that the almost complex structures $J_n$ satisfy: (1) they are invariant under the
$\R$ and $\Z_{n}$ actions on $\R\times\R/n\Z\times D$, and (2) they each lift to the same
almost complex structure $J_\infty$ on the universal covering $\R\times\R\times D$.
\end{proof}

Combining the last three statements we can prove the following.

\begin{theorem}\label{T:compact_sequences}
Let $\ubar_{n}:[S_{n},\infty)\times\R\rightarrow\R\times Z_{\infty}$ be a sequence of
$J_{\infty}$-holomorphic maps over $n\in\N$, where each $\ubar_{n}$ is a lift of a solution $\u_{n}$
to (\ref{E:hol_curve_equation_with_boundary_4}).  Suppose that $\ubar_{n}(0,0)$ is uniformly bounded in
$n$ and that $E_{\lambda}(\u_n)<\infty$ for each $n\in\N$,
and that $\lim_{n\rightarrow\infty}E_{\omega}(\u_n)=0$.  Then there exists a subsequence
$\{\u_{n_{j}}\}_{j\in\N}$ such that $\ubar_{n_{j}}$ converges in
$C^{\infty}_{\loc}(\C,\R\times Z_{\infty})$
to a $J_\infty$-holomorphic map $\u_{\infty}$ having domain either $\Sigma=\C$ or
$\Sigma=[S,\infty)\times\R\subset\C$ for some $S\in(-\infty,0]$.  Moreover $\u_\infty$
takes the following form:
there exist constants $a_{0},\tau_{0}\in\R$ such that
\begin{equation}\label{E:limit_map}
 \begin{aligned}
		  \u_\infty:&\Sigma\rightarrow\R\times\R\times D\\
		\u_{\infty}(s,t)&=(s+a_{0},t+\tau_{0},\gamma(t))
 \end{aligned}
\end{equation}
for all $(s,t)\in\Sigma$, where $\gamma\in C^\infty(\R, D)$ satisfies
\begin{align}\label{E:trajectory}
			\gammadot(t)=X_{H_{t+\tau_{0}}}(\gamma(t))
\end{align}
for all $t\in\R$.
\end{theorem}
\begin{proof}
Taking a subsequence we may assume that $\u_{n_{j}}(0,0)$ converges, and that
$S_{n_{j}}$ converges to some $S\in(-\infty,0]\cup\{-\infty\}$.
From proposition \ref{T:uniform_grad_bounds} and corollary
\ref{C:uniform_global_bds_in_Cinfty} we obtain uniform bounds on $\|\nabla\ubar_{n}\|_{C^{k}}$
for each $k\in\N$, and therefore also $C^{0}$-bounds on $\ubar_{n}$ on compact subsets, uniform in $n$.
Repeated use of the Arzela-Ascoli theorem yields a subsequence converging uniformly with all
derivatives on each compact subset of $\C$ to a smooth map
$\u_{\infty}:\Sigma\rightarrow\R\times Z_{\infty}$ where $\Sigma=[S,\infty)\times\R$ if $S$ is
finite and $\Sigma=\C$ otherwise.
From lemma \ref{L:form_of_curves_with_bounded_grad} each
map in the sequence $\ubar_{n_{j}}$ takes the form
\[
					\ubar_{n_{j}}(s,t)=(s+a_{j},t+\tau_{j},z_{j}(s,t))
\]
for constants $a_{j},\tau_{j}\in\R$, with $z_{j}:[S_{n_{j}},\infty)\times\R\rightarrow D$ satisfying
\[
				\partial_sz_{j}(s,t)+i\Big(\partial_tz_{j}(s,t)-X_H(t+\tau_j,z_{j}(s,t))\Big)=0
\]
for all $(s,t)\in[S_{n_{j}},\infty)\times\R$.  Therefore $\u_{\infty}$ takes the form
\[
			\u_{\infty}(s,t)=(s+a_{0},t+\tau_{0},z_{\infty}(s,t))
\]
for constants $a_{\infty},\tau_{\infty}\in\R$ and some
$z_{\infty}:\Sigma\rightarrow D$ satisfying
\[
		\partial_sz_{\infty}+i\Big(\partial_tz_{\infty}-X_H(t+\tau_\infty,z_{\infty})\Big)=0.
\]
Let $\omega_{\infty}:=dx\wedge dy+d\tau\wedge dH$ on $Z_{\infty}$.   Then
\[
	0\leq\int_{\R^{2}}\u_{\infty}^{*}\omega_{\infty}\leq\lim_{j\rightarrow\infty}E_{\omega}(\u_{n_{j}})=0.
\]
Thus
\begin{align*}
	&\frac{1}{2}\iint\left|\frac{\partial z_\infty}{\partial s}(s,t)\right|^{2}+
	\left|\frac{\partial z_\infty}{\partial t}(s,t)-X_{H}(t+\tau_\infty,z_{\infty}(s,t))\right|^{2}dsdt=\\
	&\hspace{3in}	\int_{\R^{2}}\u_{\infty}^{*}\omega_{\infty}=0.
\end{align*}
Hence $z_{\infty}(s,t)=\gamma(t)$ for some solution $\gamma:\R\rightarrow D$ to (\ref{E:trajectory}).
\end{proof}

\section{Construction of the finite energy foliations}\label{S:pf_of_fef}

In this final section we give a terse proof of theorem
\ref{T:existence_and_uniqueness_of_folns_for_pseudorotations}.
The approach is along standard lines, the only
part that has some small surprise is due to the presense of the boundary of the almost complex manifold.
A more general construction will appear in \cite{Bramham}.

We will assume more familiarity with terminology from \cite{BEHWZ_SFT_compactness} than 
elsewhere in this article, and with the homotopy invariant generalized intersection number for punctured
pseudoholomorphic curves in \cite{S_intersections}.

We will prove the statement of theorem
\ref{T:existence_and_uniqueness_of_folns_for_pseudorotations} for $n=1$ only, as the proof of the
general case is the same.\footnote{For maps more general than irrational pseudo-rotations the proof
for $n>1$ is a little more involved as the almost complex structure has additional symmetry
which makes transversality less obvious.  But for irrational pseudo-rotations automatic
transversality suffices and the same proof works for all $n\geq1$.}
We begin then by recalling the statement of theorem 
\ref{T:existence_and_uniqueness_of_folns_for_pseudorotations} when $n=1$.

\begin{theorem}\label{T:existence_and_uniqueness_of_folns_for_pseudorotations_appendix}
Let $H\in C^{\infty}(\R/\Z\times D,\R)$ be a Hamiltonian generating an irrational
pseudo-rotation $\varphi$.  Let $(Z=\R/\Z\times D,R)$ be the corresponding Hamiltonian
mapping torus.  Let $\gamma:\R/\Z\rightarrow Z$ be the unique $1$-periodic
orbit of $R$ for which $\gamma(0)\in\{0\}\times D$.  Assume $H$ was chosen so that $\gamma(t)=(t,0)$
for all $t\in\R/\Z$.  Let $\alpha:=\Rot(\varphi;H)\in\R$, which is necessarily irrational.

Then there exist two foliations $\F^+,\F^-$ of $\R\times Z$ by smoothly
embedded surfaces with the following properties:
\begin{itemize}
 \item \textbf{Cylinder leaf:} The cylinder $C:=\R\times\gamma(\R/\Z)\subset\R\times Z$ is a leaf in both
 $\F^+$ and $\F^-$.
 \item \textbf{Pseudo-holomorphic:} If $F\in\F^+$ (resp. $F\in\F^-$) is not $C$, then $F$ is parameterized by
 a solution $\u$ to (\ref{E:hol_curve_equation_1}) with respect to $J_{n}$ as in 
 (\ref{E:defn_of_Jtilde}), with $E_\lambda(\u)+E_\omega(\u)<\infty$
 and boundary index $\lfloor \alpha\rfloor$  (resp. $\lceil \alpha\rceil$).
 \item \textbf{$\R$-invariance:} If $F\in\F^+$ (resp. $F\in\F^-$) is a leaf and $c\in\R$, the set
 $F+c:=\{(a+c,\tau,z)\,|\,(a,\tau,z)\in F \}$ is also a leaf in $\F^+$  (resp. in $\F^-$).
 \item \textbf{Uniqueness:} $\F^+$ and $\F^-$ are uniquely determined by the above properties.
 \item \textbf{Smooth foliation:} $\F^+$ and $\F^-$ are $C^\infty$-smooth foliations at each point on the
  complement of $C$.
\end{itemize}
\end{theorem}

First we explain the construction of $\F^{+}$, and then how the approach to
constructing $\F^{-}$ differs.

\subsection{Construction of $\F^{+}$ when $\varphi|_{\partial D}$ is conjugate to a rigid rotation}\label{S:proof_with_bc}
In this section suppose that there exists $g\in\Diff^{\infty}_{+}(\partial D)$ so that
\begin{equation}
					\sigma\circ\varphi|_{\partial D}\circ\sigma^{-1}=R_{2\pi\alpha}
\end{equation}
where $R_{2\pi\alpha}:\partial D\rightarrow\partial D$ is the rigid rotation $z\mapsto ze^{2\pi i\alpha}$.

Let $H_-,H_+\in C^\infty(Z,\R)$ be as follows.  $H_{-}=H$ where $H$ is as in the statement of 
theorem \ref{T:existence_and_uniqueness_of_folns_for_pseudorotations_appendix}.  In particular the 
corresponding closed loop of Hamiltonians on the disk $H_-^t:=H_-(t,\cdot):D\rightarrow\R$ has 
time-one map $\varphi$, and $\Rot(\varphi,H_-)=\alpha\in\R$.  
Define $H_+$ by
\begin{equation}\label{E:constant_in_Hplus}
 			  H_+(\tau,z):=\pi\alpha|z|^2 + C
\end{equation}
for some constant $C\leq0$ chosen so that $\max H_+<\min H_-$.  Then the two pairs of differential forms
$\Ham_{\pm}=(\omega_{\pm},\lambda_{\pm})$, on $Z$, given by
\[
 	\omega_{\pm}=dx\wedge dy + d\tau\wedge dH_{\pm}\qquad \lambda_{\pm}=d\tau
\]
define stable Hamiltonian structures on $Z$.  Let $R_+$ and $R_-$ denote the associated
(stable Hamiltonian) Reeb vector fields, defined by $\omega_{\pm}(R_\pm,\cdot)=0$
and $\lambda_\pm(R_\pm)=1$.  These are found to be
\[
		      R_\pm(\tau,z)=\partial_\tau + X_{H_\pm^\tau}(z)
\]
where $X_{H_\pm^\tau}(z)$ is the Hamiltonian vector field associated to
$H_\pm^\tau=H_\pm(\tau,\cdot):D\rightarrow\R$ and the symplectic form $\omega_0=dx\wedge dy$.
Choose a function $H\in C^\infty(\R\times Z,\R)$ interpolating between $H_-$ and $H_+$ so that
\[
	\left\{\begin{aligned}
			&H(a,m)=H_{+}(m)& &\mbox{for all }a\geq1\\
			&\partial_{a}H(a,m)<0& &\mbox{for } -1<a<1\\
			&H(a,m)=H_{-}(m)& &\mbox{for all }a\leq-1.
		\end{aligned}\right.
\]
For example a good choice is $H(a,m)=\chi(a)H_{+}(m)+(1-\chi(a))H_{-}(m)$ for some 
$\chi\in C^{\infty}(\R,[0,1])$ with $\chi\equiv0$ on $(-\infty,-1]$ and
$\chi\equiv1$ on $[1,\infty)$ and such that $\chi'(a)>0$ for all $a\in(-1,1)$.  
 
Define an almost complex structure $\Jhat$ on $W:=\R\times Z$ by
\[
 	\left\{\begin{aligned}
		&\Jhat(a,\tau,z)\partial_\R=\partial_\tau + X_{H_{a}^{\tau}}(z) \\
		&\Jhat|_{TD}=i
		\end{aligned}\right.
\]
where for $(a,\tau)\in\R\times\R/\Z$, $X_{H_{a}^{\tau}}$ is the Hamiltonian vector
field on $(D,\omega_0=dx\wedge dy)$ for the Hamiltonian function
$H_{a}^{\tau}:=H(a,\tau,\cdot):D\rightarrow\R$.
Then $(W,\Jhat)$ is an almost complex manifold with cylindrical ends
$E_+=[1,\infty)\times Z$  and $E_-=(-\infty,-1]\times Z$, adjusted to the stable
Hamiltonian structures  $\Ham_{\pm}=(\omega_{\pm},\lambda_{\pm})$ on these ends, while
on the region $(-1,1)\times Z$, $\Jhat$ tames the symplectic $2$-form
\[
		      \Omega:=dx\wedge dy + d\tau\wedge dH.
\]
Moreover, the cylinder $C:=\R\times\gamma(\R/\Z)$ is a $\Jhat$-holomorphic curve, even though 
$\Jhat$ is not everywhere $\R$-invariant.  This is because the unique $1$-periodic orbits of 
$R_+$ and $R_-$ are the same as parameterized closed loops, and the expression for $H$ in terms 
of $H_{-},H_{+},\chi$.

Finally, we will show in section \ref{S:foliating_boundary} that the boundary of $W$, that is 
$\R\times\partial Z$, is filled by a set $\S$ of immersed $\Jhat$-holomorphic planes which, 
in the ends $E_\pm\cap(\R\times\partial Z)$ coincide with the product of the $\R$-component 
and a Reeb trajectory of $R_{\pm}$.  \\

\noindent We proceed to construct $\F^+$ in four steps.

\textbf{Step 1:} The almost complex structure $\Jhat$ on $W$ satisfies $\Jhat|_{E_+}=J_+|_{E_+}$
where $J_+$ is the cylindrical almost complex structure
\begin{equation}\label{E:acs_on_positive_end}
	\left\{\begin{aligned}
				&J_+\partial_{\R}=\partial_\tau+2\pi\alpha\partial_{\theta}\\
				&J_+|_{TD}=i
		\end{aligned}\right.
\end{equation}
on $\R\times Z$, in standard polar coordinates $(r,\theta)$ on the disk.
For each $c\in\R$ and $z\in\partial D$, the map
\begin{align}
	\u_{c,z}:&\R^{+}\times\R/\Z\rightarrow\R\times Z\label{E:ucz}\\
  	\u_{c,z}(s,t)&=(s+c,t,ze^{2\pi(\lfloor\alpha\rfloor-\alpha)s}e^{2\pi i\lfloor\alpha\rfloor t})\nonumber
\end{align}
is $J_{+}$-holomorphic.  The combined images of these maps along with
the cylinder $C:=\{(a,\tau,0,0)\,|\,a\in\R,\ \tau\in\R/\Z\}$ defines an
$\R$-invariant finite energy foliation for $(\R\times Z,J_{+})$ with boundary index
$\lfloor\alpha\rfloor$.  This is the model foliation from which we homotope.

\textbf{Step 2:}
We return to the manifold $(W=\R\times Z,\Jhat)$ with cylindrical ends.  Let $\M$ denote the moduli space
of all finite energy $\Jhat$-holomorphic curves $F\subset W$ which admit a $\Jhat$-holomorphic
parameterization by a map $\u=(a,\tau,z)\in C^{\infty}(\R^{+}\times\R/\Z,\R\times Z)$ satisfying
\[
 \left\{\begin{aligned}
  		& \u(0,\cdot)\in L_c\mbox{ and }\u\mbox{ meets }\R\times\partial Z\mbox{ transversely}\\
  		& \tau(0,\cdot):\R/\Z\rightarrow\R/\Z \mbox{ has degree }+1\\
		& z(0,\cdot):\R/\Z\rightarrow\partial D \mbox{ has degree }\lfloor\alpha\rfloor\\
 \end{aligned}\right.
\]
for any $c\in\R$, and equip $\M$ with
the topology coming from $C^{\infty}_{\loc}\cap C^0([0,\infty)\times\R/\Z,W)$ convergence.
$\M$ is non-empty as it contains the image of each curve $\u_{c,z}$ from (\ref{E:ucz})
provided we choose $c\geq1$, as $\Jhat=J_+$ on the positive end $E_+=[1,\infty)\times Z$.

Each curve $F\in\M$ is embedded, and any two curves $F_1,F_2\in\M$ are equal or disjoint.
Indeed, any two curves in $\M$ are homotopic through half cylinders with so called asymptotically
cylindrical ends (in the sense of \cite{S_intersections}).  
It therefore suffices to know that there exists a single curve $F_0\in\M$ satisfying
\[
		  F_0\cdot F_0=0 \qquad\mbox{and}\qquad F_0\cdot C=0
\]
where $C$ is the cylinder above.  This is easily verified for the explicit
curves in (\ref{E:ucz}).  Then homotopy invariance of $\cdot$ implies
that for any $F\in\M$ we have $F\cdot F=0$ and $F\cdot C=0$.
These imply $F$ is embedded by
the adjunction formula in \cite{S_intersections}.  Finally, for any two curves $F_1,F_2\in\M$ we have
$F_1\cdot F_2=F_1\cdot F_1=0$ and so either $F_1=F_2$ or $F_1\cap F_2=\emptyset$.

As a solution to the Cauchy-Riemann equations with ``free''
boundary conditions (meaning that $c\in\R$ above is not fixed) each
curve in $\M$ is Fredholm with index $2$.  Standard automatic
transversality arguments apply as for example in \cite{HWZ_propIII,Wendl_ot} and so
each curve is Fredholm regular despite the non-generic choice of almost complex structure.

\begin{lemma}
The set $\mathcal{E}:=\{w\in W\,|\, w\in F\, \mbox{for some }F\in\M\}$ is an open and closed subset
of $W\backslash C$.
\end{lemma}
\begin{proof}
Openness: as each curve $F$ in $\M$ is regular and transverse to the boundary of $W$, an implicit function
theorem as in \cite{HWZ_propIII} applies (or more precisely a version in \cite{Wendl_boundary}),
and all curves in $\M$ sufficiently close to $F$
correspond to sections of the kernel of the linearized Cauchy-Riemann operator at $F$.
Non-trivial elements of this kernel have no zeros.  This is enough to show that there is an
open neighborhood of $F$ filled by curves coming from the implicit function theorem which by construction
lie in $\M$.

Closedness:  there is a uniform bound on the total energy of all curves in $\M$.
Thus for any sequence $F_k\in\M$, see the discussion above, symplectic field theory compactness as in
\cite{BEHWZ_SFT_compactness} applies.  We have curves with boundary, which is not treated by
\cite{BEHWZ_SFT_compactness}, but this is not a problem due to the filling $\S$.  Suppose that
$p_k\in\R\times Z$ is a sequence of points in $\mathcal{E}$,
with $p_k\rightarrow p_\infty$ some point $p_\infty\notin C$.  Let $F_k\in\M$ be a curve containing $p_k$.
Viewing these as marked points on the curves $F_k$ we take a limit in the SFT sense.   The unique stable
limiting building $\Fbar$ has a component $F_\infty$ with a marked point, this component must contain
$p_\infty$.  It is easy to see that the component $F_\infty$ must be a half cylinder in $\M$ provided
it is not the cylinder $C$ (see section \ref{S:remarks_on_compactness} for a brief justification of this).
But $F_\infty$ cannot equal $C$ as $p_\infty\notin C$.
\end{proof}

We conclude that there is a unique curve in $\M$ going through each point
in $W\backslash C$.

\textbf{Step 3:}
Let $p_k\in W\backslash C$ be any sequence of points which converge to a point $p_\infty\in C$ on
the distinguished cylinder.  From the previous step there exists then a sequence of curves
$F_k\in\M$ with $p_k\in F_k$.  Applying the SFT compactness theorem to this sequence, where
$p_k$ is viewed as the image of a marked point of $F_k$, we get
convergence to a stable nodal holomorphic building $\Fbar$ say, with a single marked point which
corresponds to $p_\infty$ and which lies on the distinguished middle level of $\Fbar$.
It follows that the middle level of $\Fbar$ is simply the cylinder $C$ as $p_\infty$ cannot
be an isolated intersection point with $C$.  Therefore
the building $\Fbar$ must have non-empty lower levels.  These lower levels are holomorphic curves
in the cylindrical manifold $(\R\times Z,J_-)$ corresponding to the negative end of $(W,\Jhat)$, which
was modelled on the pseudo-rotation $\varphi$.
As the building is stable and has no marked points in the lower levels, there is
precisely one lower level in $\Fbar$ and it must correspond to
a half cylinder with positive puncture asymptotic to the unique simply covered periodic orbit $\gamma_-$
of the Reeb flow of $\Ham_-$.

Denote this half cylinder by $G_0$.  Intersection considerations show that
since $C\cdot F_k=0$ and $F_k\cdot F_k=0$ for all $k$, and these curves only break at elliptic orbits,
we must have
\[
		  G_0\cdot G_0=0 \qquad\mbox{and}\qquad G_0\cdot C_-=0
\]
where $C_-$ is the unique orbit cylinder in $(\R\times Z,J_-)$.  It follows as in step 2 that the
half cylinder $G_{0}$ is embedded.

\textbf{Step 4:} The embedded half cylinder $G_0$ in the cylindrical manifold $(\R\times Z,J_-)$ found at the end
of the last step is automatically Fredholm regular and has Fredholm index $2$.  We may now argue in the
manner of step 2, and conclude $G_0$ lies in a non-empty $2$-dimensional moduli space of half cylinders
$\M_-$
which fill an open and closed subset of $(\R\times Z)\backslash C_-$.  Each $G\in\M_-$ is homotopic to
$G_0$ through curves with asymptotically cylindrical ends, and we can conclude that
\[
		  G\cdot G=0 \qquad\mbox{and}\qquad G\cdot C_-=0
\]
holds for all $G\in\M_-$.  This implies that:
\begin{enumerate}
 \item Each $G\in\M_-$ is embedded.
 \item Each pair of curves $G_1,G_2\in \M_-$ is either equal or disjoint.
\end{enumerate}
The second of these also implies that if $G\in\M_-$ and $c\in\R\backslash\{0\}$, then
the curve translated in the $\R$-direction $G+c:=\{(a+c,m)\,|\,(a,m)\in G\}$ (which is
in $\M_-$ as $J_-$ is $\R$-invariant) is disjoint from $G$ because $G$ and $G+c$
cannot be equal as their boundaries are in different totally real submanifolds.

\subsection{Construction of $\F^{+}$ without boundary restrictions}\label{S:pf_without_bc}
This completes the construction for any irrational pseudo-rotation $\varphi:D\rightarrow D$
which restricts to a circle diffeomorphism $\varphi|_{\partial D}$ that is smoothly conjugate to
a rigid rotation.  
By a deep result of Herman the set of such pseudo-rotations is dense  
amongst all irrational pseudo-rotations, and so a further limiting step can remove the 
boundary restrictions.   We explain this now. 

Say that an irrational number $\alpha$ belongs to $\Diophantine$ if it satisfies the following
Diophantine condition: there exists $n\geq 2$ and $C\in(0,\infty)$ such that for all $(p,q)\in\Z\times\N$
\[
		  \left|\alpha - \frac{p}{q} \right|\geq C\frac{1}{q^n}.
\]
It is easy to show that $\Diophantine$ is dense in $\R$ (indeed of full measure).
Denote by $\Diff_{\Diophantine}(\partial D)$ those orientation preserving $C^{\infty}$-smooth 
circle diffeomorphisms that have rotation number in $\Diophantine$.

\begin{theorem}[Herman \cite{Herman}]
If $f\in\Diff_\Diophantine(\partial D)$ then there
exists $\sigma\in\Diff^\infty_+(\partial D)$ such that $\sigma^{-1}f\sigma=R$
where $R:\partial D\rightarrow\partial D$ is a rigid rotation.
\end{theorem}

\begin{lemma}
$\Diff_{\Diophantine}(\partial D)$ is dense in $\Diff^\infty_+(\partial D)$ with the $C^\infty$-topology.
\end{lemma}
\begin{proof}
Fix any $f\in\Diff_{\Diophantine}(\partial D)$.  Consider the continuous path
$f_t:=R_{2\pi t}\circ f\in\Diff^\infty_+(\partial D)$ over $t\in[0,1/2]$.  The rotation
numbers $\Rot(f_t)$ vary continuously with $t$.  Moreover, there is a continuous family of lifts
$\ftilde_t:\R\rightarrow\R$ with the monotonicity property that for $t>0$,
$\ftilde_t(x)>\ftilde_0(x)$ for all $x\in\R$.  Therefore, since $\Rot(f_0)$ is \emph{irrational},
it follows that for all $t\in(0,1/2]$, $\Rot(f_t)>\Rot(f_0)$.  See for example proposition
11.1.9 in \cite{Katok_H}.  As $\Diophantine$ is dense in $\R/\Z$, we find a sequence
$t_j\in(0,1/2]$ converging to zero such that $\Rot(f_{t_j})\in\Diophantine$.
\end{proof}

Now suppose that $\varphi\in\Diff^\infty(D,\omega_0)$ is any smooth irrational pseudo-rotation.
Let $H\in C^\infty(\R/\Z\times D,\R)$ be a Hamiltonian with time-one map $\varphi$.
Using the lemma one can find smooth perturbations $H_j\in C^\infty(\R/\Z\times D,\R)$ of $H$ near
the boundary $\R/\Z\times\partial D$, so that
\begin{align*}
			    & H_j\rightarrow H \mbox{ in }C^\infty\\
			    &\Rot(\varphi_j|_{\partial D})\in\Diophantine
\end{align*}
for all $j$, where $\varphi_j$ is the time-one map of $H_j$.  For $j$ sufficiently large $\varphi_j$ has
no fixed points besides the origin. (It may not be a pseudo-rotation but that does not matter.) 
By Herman's theorem $\varphi_j|_{\partial D}$
is smoothly conjugate to a rigid rotation.  Therefore, by steps 2 to 4 in section
\ref{S:proof_with_bc} we can find a finite energy
foliation $\F_{H_j}$ of the cylindrical almost complex manifold $(\R\times Z,J_{H_j})$.

The almost complex structures $J_{H_j}$ converge uniformly to $J_H$ as $j\rightarrow\infty$, 
and we may pass to a limit of foliations and obtain a finite energy foliation for $(\R\times Z,J_{H})$.
For example one can take a limit of a single sequence of half cylinder leaves $F_j\in\F_{H_j}$ and obtain a
single $J_{H}$-holomorphic half cylinder $F_\infty$ disjoint from the orbit cylinder $C$ and with vanishing
self intersection number.  The moduli space containing $F_\infty$ fills up the complement of
$C$ by embedded curves that are pairwise identical or disjoint as we require.  This completes
the construction of $\F^{+}$ in theorem
\ref{T:existence_and_uniqueness_of_folns_for_pseudorotations_appendix}.

\subsection{Uniqueness}
That the foliation $\F^{+}$ is uniquely determined by the other properties in theorem 
\ref{T:existence_and_uniqueness_of_folns_for_pseudorotations_appendix}
is more or less immediate from the intersection theory 
because there is only one $1$-periodic orbit and all curves with boundary have the same boundary index.

\subsection{Smoothness}
The proof that $\F^{+}$ is a smooth foliation on $(\R\times Z)\backslash C$ is the same 
as that used in Hofer-Wysocki-Zehnder \cite{HWZ_fef}, based on their implicit function theorem 
in \cite{HWZ_propIII}.   

By smoothness of $\F^+$ at a point $p\in\R\times Z$ we mean that there exists a smooth
local diffeomorphism $\sigma:U\rightarrow V\subset\R^{2}\times\R^{2}$ on a neighborhood
of $p$ onto an open convex subset $V\subset\R^{4}$, such that each leaf $F\in\F^+$ having non-empty
intersection with $U$ is mapped under $\sigma$ onto a subset of the form
$V\cap(\R^{2}\times\{c\})$ for some $c\in\R^{2}$.

The point is that 
each leaf $F\in\F^+$, with $F\neq C$, is embedded, meets the boundary of $\R\times Z$ transversally in
some totally real submanifold, has Fredholm index $2$, and is Fredholm regular.   An implicit
function theorem as generalized in \cite{Wendl_boundary} applied to $F$ gives a neighborhood
of $F$ in $(\R\times Z,J)$ filled by $J$-holomorphic curves, given by exponentiating
sections of a normal bundle over $F$.  (With respect to a metric 
for which each $L_{c}$, and $\R\times\partial Z$, are totally geodesic.)  
These curves from the implicit function theorem form a $C^\infty$-smooth
local foliation $\Ftilde$ of a neighborhood of $F$ in $\R\times Z$.  Intersection theory considerations
show that each leaf in $\Ftilde$ is also a leaf in the global foliation $\F^+$.  So $\F^{+}$ is 
also smooth at points on $F$.

\subsection{Constructing $\F^{-}$}
The construction of $\F^-$ is along exactly the same lines as for $\F^+$, but we start with
a different model foliation in step 1.  Indeed, in place of the curves in (\ref{E:ucz}),
each of which is bounded from below, we use curves of the form
\begin{align}
	v_{c,z}:&\R^-\times\R/\Z\rightarrow\R\times Z\label{E:ucz_2}\\
  	v_{c,z}(s,t)&=(s+c,t,ze^{2\pi(\lceil\alpha\rceil-\alpha)s}e^{2\pi i\lceil\alpha\rceil t})\nonumber
\end{align}
over $c\in\R$ and $z\in\partial D$.  (These maps are also pseudoholomorphic with respect to the
almost complex structure in (\ref{E:acs_on_positive_end}).)  Each $v_{c,z}$ has image bounded from above,
so we have to insert one of them into the negative end of our
almost complex manifold $(W=\R\times Z,\Jhat)$.  So the main difference is that from the start we
reverse the roles of $H_-$ and $H_+$, this time picking a positive constant $C$ in
(\ref{E:constant_in_Hplus}) so that $\min H_->\max H_+$ still holds.  Then the remaining steps
are exactly analogous, and in the final foliation the curves with boundary have boundary index
$\lceil\alpha\rceil$ instead, as the curves in (\ref{E:ucz_2}) do.

\subsection{Foliating the boundary}\label{S:foliating_boundary}
We need to justify the existence of the foliation $\S$ of the boundary of $\R\times Z$
that was used in section \ref{S:proof_with_bc}.

For each $a\in\R$, consider for a fixed value of $a$ the resulting time-dependent Hamiltonian
on the disk $H_a$ given by
\[
			H_{a}^{t}:=H(a,t,\cdot):D\rightarrow\R
\]
over $t\in\R/\Z$.  By modifying $H_a$ on any arbitrarily small neighborhood of the boundary of the disk
we may arrange that the time-one map of the path of generated Hamiltonian disk maps is
any prescribed orientation preserving diffeomorphism on the boundary of the disk.
By extension, given any smooth path $a\mapsto f_a\in\Diff^\infty_+(\partial D)$
over $a\in\R$, satisfying
\begin{equation}\label{E:path_of_boundary_condns}
 	f_a=\left\{\begin{aligned}
	            &R_{2\pi\alpha}|_{\partial D}& &\mbox{if }a\geq1\\
	            &\varphi|_{\partial D}& &\mbox{if }a\leq-1,
	           \end{aligned}\right.
\end{equation}
the function $H:\R\times Z\rightarrow\R$ may be modified on any small neighborhood of $[-1,1]\times\partial Z$
so that for each $a\in\R$ the time-one map of the modified Hamiltonian $H_{a}$ now coincides with
$f_a$ on the boundary of the disk.

Suppose that we can find a smooth path $a\mapsto f_a\in\Diff^\infty_+(\partial D)$ satisfying
(\ref{E:path_of_boundary_condns}), and which additionally has the property that each $f_a$
is smoothly conjugate to the rigid rotation $R_{2\pi\alpha}|_{\partial D}$.  More precisely,
suppose that we find a smooth map $g\in C^\infty(\R\times\partial D,\partial D)$, so that
for each $a\in\R$ the map $g_a:=g(a,\cdot)$ is an element of $\Diff^\infty_+(\partial D)$,
and with the property that the path
$f_a:=g_aR_{2\pi\alpha}g_a^{-1}$
satisfies (\ref{E:path_of_boundary_condns}).  Then we may modify
$H\in C^\infty(\R\times Z,\R)$ near $[-1,1]\times\partial Z$
so that for each $a\in\R$ the time-one map of the time-dependent Hamiltonian $H_{a}:=H(a,\cdot)$
equals $f_a$ on the boundary of the disk.  For each $a\in\R$ let
\begin{align*}
			\phi_a:&\R\times\partial Z\rightarrow\partial Z \\
			      &(t,z)\mapsto\phi_a^t(z)
\end{align*}
denote the $1$-parameter family of maps generated by $X_{H_{a,t}}$ on $\partial D$.
So in particular $\phi_a^1=f_a$ for all $a\in\R$.  Now for each $z\in\partial D$,
\[
	S_z:=\Big\{\big(a,t,g_a(\phi_a^t(z))\big)\in\R\times\R/\Z\times\partial D\,|\,a\in\R,\ t\in\R \Big\}
\]
is an immersed surface in $\R\times\partial Z$, and the union
\[
	      \S:=\bigcup_{z\in\partial D} S_z
\]
is a foliation of $\R\times\partial Z$.  As $\alpha$ is irrational each $S_z$ is dense in $\R\times\partial Z$.
However, the relation
\begin{equation}\label{E:special_form_of_each_fa}
 			    f_a:=g_aR_{2\pi\alpha}g_a^{-1}
\end{equation}
for all $a\in\R$ enables us to find a $C^\infty$-smooth almost complex structure $J'$
on $\R\times Z$, prescribed at points on $\R\times\partial Z$ so that each $S_z$ has $J'$-invariant tangent bundle.
Indeed, differentiating the expression $\big(a,t,g_a(\phi_a^t(z))\big)$ in $a$ gives a vector field $V_1$ say,
while differentiating it in $t$ results in a vector field $V_2$.  Both are non-vanishing and transverse as we will
see, so we can set $J'V_1=V_2$.  That $V_1$ is indeed a well defined vector field uses
(\ref{E:special_form_of_each_fa}).  Moreover, one finds that:
\[
		  V_2=\partial_\tau + X_{H_{a}^{\tau}},
\]
while
\[
		    V_1=\partial_a + V_3,
\]
for some $V_3$ on $\R\times\partial Z$ that has no $\partial_a$ component, and tends to zero
in $C^0$ as $\|\partial_af_a\|_{C^0}$ tends to zero.  We can arrange that $\|\partial_af_a\|_{C^0}$
is as small as we wish by ``slowing everything down'', that is, replacing the interval $[-1,1]\times Z$
by $[-N,N]\times Z$ for sufficiently large $N>0$.  Then, from these expressions for $V_1,V_2$ we see
that $J'$ extends to an almost complex structure $J''$ on $\R\times Z$ with the following properties
if $\|\partial_af_a\|_{C^0}$ is sufficiently small:
\begin{enumerate}
 \item $J''$ coincides with the almost complex structure $J$ outside of a
 small neighborhood of $[-N+1,N-1]\times\partial Z$.
 \item Each surface $L_c:=\{c\}\times\partial Z$ is totally real with respect to $J''$.
 \item $J''$ is tamed by the symplectic form $\Omega$ on $(-N,N)\times Z$.
\end{enumerate}
And finally of course $\S$ is a $J''$-holomorphic filling of the boundary $\R\times\partial Z$.

The only remaining question is when the relation (\ref{E:special_form_of_each_fa}) can be
arranged for all $a\in\R$.  But this holds if and only if the circle
maps $\varphi|_{\partial D}$ and $R_{2\pi\alpha}|_{\partial D}$ are conjugate by
an orientation preserving $C^\infty$-smooth diffeomorphism.  Necessity is obvious, let us
show sufficiency.
Suppose that there exists $g\in\Diff^\infty_+(\partial D)$ such that
$\varphi|_{\partial D}=gR_{2\pi\alpha}|_{\partial D}g^{-1}$.  Since $g$ has degree $+1$
it is smoothly isotopic to the identity and we may find a smooth isotopy $g_a\in\Diff^\infty_+(\partial D)$
over $a\in\R$, satisfying $g_a=\id$ for all $a\geq 1$, and $g_a=g$ for all $a\leq -1$.
Thus the smooth path $f_a\in\Diff^\infty_+(\partial D)$ over $a\in\R$ defined by
$f_a:=g_aR_{2\pi\alpha}g_a^{-1}$ satisfies
\[
	f_a=\left\{\begin{aligned}
	            &R_{2\pi\alpha}& &\mbox{if }a\geq1\\
	            &\varphi|_{\partial D}& &\mbox{if }a\leq-1
	           \end{aligned}\right.
\]
and therefore has the properties we require.

\subsection{Compactness}\label{S:remarks_on_compactness}
In section \ref{S:proof_with_bc} we applied, a number of times, the compactness theory in 
\cite{BEHWZ_SFT_compactness} that was developed for symplectic field theory, to sequences 
of punctured pseudoholomorphic curves in $(W=\R\times Z,\Jhat)$.  

Although $\partial W=\R\times\partial Z$ is non-empty, 
we arranged in the last section that it can be foliated by immersed $\Jhat$-holomorphic curves 
without boundary, in a reasonably nice way (meaning that they lift to properly embedded curves on 
the universal covering of $\partial W$).  

Consider a sequence $F_k$ of $\Jhat$-holomorphic curves in $(W,\Jhat)$ which are half cylinders
with totally real boundary conditions $\partial F_k\subset L_{c_k}$ for some $c_k\in\R$.  Suppose
that for each $k$ the generalized intersection number, in the sense of \cite{S_intersections}, with
the $\Jhat$-holomorphic cylinder $C$ is zero; $F_k\cdot C=0$ for all $k$.  Furthermore, suppose that
each $F_k$ meets the boundary of $W$ transversely.  Then we used the following several times: 

\begin{proposition}
If each $F_k$ has at most one marked point, and the sequence $F_k$ converges to a stable nodal
holomorphic building $\Fbar$ in the sense of \cite{BEHWZ_SFT_compactness}, then $\Fbar$ 
has the following form:
\begin{itemize}
 \item At most one component in each level.  
 \item No nodal points or closed curves.  
 \item One component with boundary, which is a half cylinder (meaning a disk with an interior puncture) 
 with the same sign puncture as that of each $F_{k}$, $k$ large, and meeting $\partial W$ transversally 
 along the boundary of the curve.  
 \item All other components of $\Fbar$, if there are any, are cylinders with punctures of opposite sign.  
 \item Interior points of $\Fbar$ are disjoint from $\partial W$.   
\end{itemize}   
\end{proposition} 
\begin{proof}
$\Fbar$ has genus zero in the generalized sense of \cite{BEHWZ_SFT_compactness} as each $F_{k}$ 
has genus zero.  So each component of $\Fbar$ has genus zero also.  

There are no planes as there are no contractible periodic orbits.  

There are also no disk components in $\Fbar$.  To see this we argue indirectly.  
Suppose that $\Fbar$ has a disk component $D$.  Then $\partial D$
must be a closed loop in one of the totally real surfaces $L_{c}\simeq\partial Z$ say.  As $\partial D$ is
contractible in $W$ it must lie in a homology class $m1_{\partial D}\in H_1(\partial Z)$
for some $m\in\Z$.  The topological count of intersections between $D$ and the cylinder $C$ is $m$.
The building is stable with at most one marked point implies that the component $D$ is not just a point.
Therefore it has non-zero energy, which implies that $m\neq 0$.  Therefore $D$ and $C$ are distinct
holomorphic curves with interior intersections.  Therefore $F_k$ and $C$
intersect for large $k$, contradicting that they are infact disjoint for all $k$.

There are no closed curves in $\Fbar$ as $\Omega$ is exact.  Also there are no nodes, as these
would lead to a component asymptotic to a contractible periodic orbit which is impossible.

It follows that the limit building has one boundary component, and so has a unique component
with non-empty boundary, and that this is a half cylinder.
Due to the orientations of the periodic orbits in the ends this half cylinder cannot have 
a different sign puncture to each $F_{k}$.  

The presence of the foliation $\S$ of $\partial W$ prevents interior intersections between components 
of $\Fbar$ and $\partial W$, and prevents the boundary of $\Fbar$ becoming tangent anywhere to $\partial W$.  
\end{proof}

\appendix

\section{Proof of proposition \ref{P:LeCalvez_observation_intro}}
In this appendix we prove the following statement, which implies proposition 
\ref{P:LeCalvez_observation_intro}.  The idea of the proof was explained to be by Patrice LeCalvez.  

We write $R_{\theta}:D\rightarrow D$ to denote the rigid rotation $z\mapsto e^{i\theta}z$
through angle $\theta\in\R$.  

\begin{proposition}\label{P:LeCalvez_observation}
Consider a sequence $\varphi_k\in\Homeo_+(D)$ converging in the $C^0$-topology to
$\varphi\in\Homeo_{+}(D)$, where all maps fix $0\in D$.
Under the following additional assumptions it follows that $\varphi$ has
no periodic points in $D\backslash\{0\}$.
\begin{enumerate}
 \item For each $k\in\N$, there exists $g_k\in\Homeo_+(D)$, fixing the origin, such that
 $\varphi_k=g_k^{-1}R_{2\pi\theta_k}g_k$, some $\theta_k\in\R$.
 \item $\theta_k\rightarrow\theta$ as $k\rightarrow\infty$, where $\theta$ is
 irrational.
\end{enumerate}
\end{proposition}

To prove this we need to recall the notions of positively and negatively returning disks due 
to Franks \cite{Franks_1}.  Let $\Ddot:=D\backslash\partial D$.  Denote by 
$A:=\Ddot\backslash\{0\}$ and $\Atilde:=(0,1)\times\R$ the open annulus and its universal covering 
via the covering map
\begin{align*}
				\pi:(0,1)\times\R&\rightarrow\Ddot\backslash\{0\}\\
							(x,y)&\mapsto xe^{2\pi iy}.
\end{align*}
Let $T:\Atilde\rightarrow\Atilde$ be the deck transformation $T(x,y)=(x,y+1)$.
By an open disk $U\subset\Atilde$ is meant an open subset homeomorphic to $\Ddot$ with the
subspace topology.

\begin{definition}
Let $f:\A\rightarrow\A$ be a homeomorphism homotopic to the identity, and
$\ftilde:\Atilde\rightarrow\Atilde$ a lift via $\pi$.
Consider an open disk $U\subset\Atilde$ for which
\begin{align}
						\ftilde(U)\cap U=\emptyset.\label{E:8989}
\intertext{If there exist integers $n>0$, $k\neq0$, such that}
						\ftilde^{n}(U)\cap T^{k}U\neq\emptyset\label{E:8988}
\end{align}
then $U$ is called a \emph{positively}, resp. \emph{negatively}, \emph{returning disk}
for $\ftilde$ if $k>0$, resp. $k<0$.
\end{definition}

\begin{remark}\label{R:disks_for_perturbations}
Consider a disk $U\subset\Atilde$ satisfying (\ref{E:8989}) and (\ref{E:8988}).
If moreover the closure of the disk satisfies (\ref{E:8989}), that is,
$\ftilde(\Ubar)\cap\Ubar=\emptyset$, then for any sufficiently $C^0$-small perturbation of $\ftilde$,
$U$ will satisfy (\ref{E:8989}) and (\ref{E:8988}) for the perturbed map also.
\end{remark}

The key result for us is the following strong generalization of the Poincar\'e-Birkhoff
fixed point theorem, theorem 2.1 in \cite{Franks_1}.

\begin{theorem}[Franks]\label{T:Franks_gen_of_PB}
Let $f:\A\rightarrow\A$ be a homeomorphism of the open annulus homotopic to the identity,
and for which every point is non-wandering.  If there exists a lift
$\ftilde:\Atilde\rightarrow\Atilde$  having a positively returning disk which is a lift
of a disk in $\A$, and a negatively returning disk that is a lift of a disk in $\A$, then $f$
has a fixed point.
\end{theorem}

Recall that a point $x\in\A$ is \emph{non-wandering for $f$} if for every open neighborhood $U$ of
$x$ there exists $n>0$ such that $f^{n}(U)\cap U\neq\emptyset$.  It is easy to prove that
every point is non-wandering for a homeomorphism $f$ if $f$ preserves a finite measure that is
positive on open sets.

We now use theorem \ref{T:Franks_gen_of_PB} to prove the proposition.

\begin{proof}[Proof of proposition \ref{P:LeCalvez_observation}]
Arguing indirectly, suppose that $\varphi$ has a periodic point $z_0\in D\backslash\{0\}$.
Then we will show that some iterate of $\varphi$, let us call it $\varphi^{n}$, has a lift
via the covering map
\begin{align*}
		\pi:(0,1]\times\R\rightarrow D\backslash\{0\}\\
		      (x,y)\mapsto xe^{2\pi iy}
\end{align*}
to a map of the half closed infinite strip
$\varphitilde^n:(0,1]\times\R\rightarrow(0,1]\times\R$, having disks $U_-,U_+\subset(0,1)\times\R$
which satisfy the following:
\begin{enumerate}
 \item $U_+$ is a positively returning disk for $\varphitilde^n$.  $U_-$ is a negatively returning disk
 for $\varphitilde^n$.
 \item The closures satisfy $\ftilde(\Ubar_+)\cap\Ubar_+=\emptyset$ and
 $\ftilde(\Ubar_-)\cap\Ubar_-=\emptyset$.
 \item $U_+$ and $U_-$ are lifts of disks in $\Ddot\backslash\{0\}$.
\end{enumerate}

Suppose for a moment that we have established this.  Then
since $\varphi_k:D\rightarrow D$ converges uniformly in the $C^0$-topology
to $\varphi$, there exists a sequence of lifts
\begin{align*}
		\varphitilde_k:(0,1]\times\R\rightarrow(0,1]\times\R
\end{align*}
of $\varphi_k$ which converges uniformly in the $C^0$-topology to $\varphitilde$.  Hence by
remark \ref{R:disks_for_perturbations} there exists
$K\in\N$ such that for all $k\geq K$, $U_+$ is a positively returning disk for $\varphitilde_k^{n}$,
and $U_-$ is a negatively returning disk for $\varphitilde_k^{n}$.

Recall that $U_+$ and $U_-$ lie in the interior $(0,1)\times\R$, and both are lifts of disks in
$\Ddot\backslash\{0\}$.  Therefore we may apply Franks' theorem, theorem \ref{T:Franks_gen_of_PB},
to each map of the open annulus $\varphi_k^{n}=(\varphi_k)^n:\Ddot\backslash\{0\}\rightarrow\Ddot\backslash\{0\}$
for $k\geq K$.  Indeed, being conjugate to a rigid rotation, every point is non-wandering for
$\varphi_k^{n}$.  We conclude that for all $k\geq K$,
$\varphi_k^{n}$ has a fixed point in $D\backslash\{0\}$.  But $\varphi_k^{n}$ is conjugate to the
rigid rotation $R_{2\pi n\theta_k}$.  So $n\theta_k\in\Z$ for all $k\geq K$, which contradicts the
convergence of the sequence $\theta_k$ to an irrational number.

With this contradiction we will be done, and so it remains to establish that some iterate of
$\varphi$ has a lift for which we can find disks $U_+,U_-\subset(0,1)\times\R$ satisfying conditions
(1), (2), and (3) listed above.

We are assuming that $\varphi$ has a periodic point besides $0\in D$.
From this it is not hard to see that for some $n>0$ sufficiently large there exists a
lift $\varphitilde^n:(0,1]\times\R\rightarrow(0,1]\times\R$ of $\varphi^n$ such that one of
the following two possibilities occurs: (1) there is a fixed point $z\in\Ddot\backslash\{0\}$
of $\varphi^n$ with a lift $(x,y)\in(0,1)\times\R$, such that $\varphitilde^n(x,y)=(x,y-1)$,
and for all $y\in\R$, $\varphitilde^n(1,y)=(1,y')$ where
$y'>y+1$.  Or (2), there is a fixed point $z\in\Ddot\backslash\{0\}$
of $\varphi^n$ with a lift $(x,y)\in(0,1)\times\R$, such that $\varphitilde^n(x,y)=(x,y+1)$,
and for all $y\in\R$, $\varphitilde^n(1,y)=(1,y')$ where
$y'<y-1$.

Let us consider the first situation, as the second is dealt with similarly.
Any sufficiently small disk neighborhood $U_-\subset(0,1)\times\R$ of $(x,y)$, satisfies
$\varphitilde^n(\Ubar_-)\cap \Ubar_-=\emptyset$ and is a lift of a disk in $\Ddot\backslash\{0\}$.
Moreover $\varphitilde^n(U_-)\cap T^{-1}(U_-)$ is non-empty as it contains the point $(x,y-1)$.
Thus $U_-$ is a negatively returning disk for $\varphitilde^n$ satisfying conditions (1), (2),
and (3) above.

It remains to find a suitable positively returning disk for $\varphitilde^n$.  Recall that
for all $y\in\R$,
\[
			      \varphitilde^n(1,y)=(1,y')
\]
where $y'>y+1$.  As $\partial D$ is a compact invariant set for $\varphi^n$, there exists a point
$z_1\in\partial D$ that is non-wandering for $\varphi^n$ (e.g. take any point in the $\omega$-limit
set of an orbit).  Let $(1,y)\in(0,1]\times\R$ be a lift of $z_1$.  Then for every sufficiently small
closed neighborhood $I\subset\partial\Atilde:=\{1\}\times\R$ of $(1,y)$ there exists $m>0$ and $k>0$ such
that
\begin{align}
			  \varphitilde^n(I)\cap I=&\emptyset\label{E:477} \\
		 \varphitilde^{nm}(I)\cap T^k(I)&\neq\emptyset\label{E:476}.
\end{align}
Taking $I$ sufficiently small we may also assume that it is a lift of an interval in $\partial D$,
that is, a closed, connected, non-empty, simply connected set.

Consider the open neighborhoods of $I$ in $\Atilde$ of the form
\[
			  V:=\{z\in(0,1]\times\R\,|\, d(z,I)<\varepsilon\}
\]
for $\varepsilon>0$.  Clearly,
\begin{equation}\label{E:5844}
					\varphitilde^{nm}(V)\cap T^k(V)\neq\emptyset
\end{equation}
from (\ref{E:476}) and is open in $(0,1]\times\R$.  And for all $\varepsilon>0$ sufficiently small
we have
\begin{equation}\label{E:5845}
			\varphitilde^n(\Vbar)\cap\Vbar=\emptyset
\end{equation}
from (\ref{E:477}).  Set $U_{+}:=V\cap\Atilde=V\cap\big((0,1)\times\R\big)$.  From (\ref{E:5844})
we get
\[
					\varphitilde^{nm}(U_{+})\cap T^k(U_{+})\neq\emptyset.
\]
From (\ref{E:5845})
\[
					\varphitilde^n(\Ubar_{+})\cap\Ubar_{+}=\emptyset
\]
as $\Ubar_+=\Vbar$.
Moreover, for $\varepsilon>0$ sufficiently small $U_{+}$ is a disk in $\Atilde$ and is a
lift of a disk in $\Ddot\backslash\{0\}$.  Thus $U_{+}$ is a positively returning disk for
$\varphitilde^{n}$ satisfying the required conditions (1), (2), and (3).
\end{proof}

\section{Nondegeneracy}

Let $\varphi:D\rightarrow D$ be an irrational pseudo-rotation with rotation number $[\alpha]\in\R/\Z$.  
In this appendix we prove the following.

\begin{lemma}\label{L:fixed_point_of_pseudorotation_is_strongly_elliptic}
 The linearization $D\varphi(0)$ has eigenvalues $\{e^{2\pi i\alpha},e^{-2\pi i\alpha}\}$.
\end{lemma}

We begin with:

\begin{lemma}\label{L:appendix_blowing_up_fixedpt}
Let $\pi:(0,1]\times\partial D\rightarrow D\backslash\{0\}$ be the $C^\infty$-diffeomorphism
$\pi(t,x):=\sqrt{t}x$.  Then the diffeomorphism
$\pi^{-1}\circ\varphi\circ\pi$ on the half open annulus $(0,1]\times\partial D$ preserves
the area form $dt\wedge d\theta$ and has an extension to a
homeomorphism on the closed annulus $\varphihat:[0,1]\times\partial D\rightarrow[0,1]\times\partial D$
given by
\begin{align*}
	      \varphihat(0,x)=\left(0\,,\frac{D\varphi(0)[x]}{|D\varphi(0)[x]|}\right).
\end{align*}
\end{lemma}
\begin{proof}
First, $\pi^*(dx\wedge dy)=\pi^*(rdr\wedge d\theta)=\sqrt{t}d(\sqrt{t})\wedge d(\theta)=(1/2)dt\wedge d\theta$.
So $\varphihat$ preserves the area form $dt\wedge d\theta$.

It remains to justify that the extension is continuous.
Define $\Phi\in C^\infty(D,\textup{gl}(\R^2))$ by
\begin{align*}
		    \Phi(z)[v]=\int_0^1 D\varphi(tz)[v]dt
\end{align*}
for all $z\in D$ and $v\in\R^2$.  As $\varphi(0)=0$, by the fundamental theorem of calculus
\begin{equation}\label{E:varphi_interms_of_Phi}
		\varphi(z)=\Phi(z)[z]
\end{equation}
for all $z\in D$.  Hence the continuous map $[0,1]\times\partial D\rightarrow\R^2$
given by
\[
		  (t,x)\mapsto \Phi(\sqrt{t}x)[x]
\]
is knowhere vanishing; on $t>0$ because of relation (\ref{E:varphi_interms_of_Phi}) and
that $\varphi$ maps only the origin to the origin; at $t=0$ because $\Phi(0)=D\varphi(0)$ has
no kernel.

A calculation shows that for all $t>0$ and $x\in\partial D$,
\begin{equation}\label{E:634}
	      \varphihat(t,x)=
	      \left(|\varphi(\sqrt{t}x)|^2,
			\frac{\varphi(\sqrt{t}x)}{|\varphi(\sqrt{t}x)|}\right).
\end{equation}
Using (\ref{E:varphi_interms_of_Phi}) we have, for $t>0$ and $x\in\partial D$,
\begin{align}\label{E:635}
  \frac{\varphi(\sqrt{t}x)}{|\varphi(\sqrt{t}x)|}=
		  \frac{\Phi(\sqrt{t}x)[x]}{|\Phi(\sqrt{t}x)[x]|}.
\end{align}
As we observed, the denominator in the right hand side is well defined, continuous, and
knowhere vanishing, even at $t=0$.  Thus the right hand side of (\ref{E:635}) defines a continuous extension
of the left hand side, and at $t=0$ takes the value
\[
		      \frac{D\varphi(0)[x]}{|D\varphi(0)[x]|}.
\]
Thus the right hand side of (\ref{E:634}) extends continuously to $t=0$,
and the extended map satisfies
\[
	  \varphihat(0,x)=\left(0\, , \frac{D\varphi(0)[x]}{|D\varphi(0)[x]|}\right).
\]
This proves lemma \ref{L:appendix_blowing_up_fixedpt}.
\end{proof}

A real linear map $T:\R^2\rightarrow\R^2$ defines a circle map by
\begin{equation}\label{E:9841}
 \begin{aligned}
	      f_{T}:& \partial D\rightarrow\partial D\\
 		&\ v\mapsto \frac{Tv}{|Tv|}.
 \end{aligned}
\end{equation}
Let us record some elementary properties of this correspondence:
\begin{enumerate}
 \item If $T$ is an orientation preserving isomorphism then $f_{T}$ is an orientation preserving homeomorphism.
 \item If $T_1$ and $T_2$ are two linear maps then $f_{T_1\circ T_2}=f_{T_1}\circ f_{T_2}$, and moreover
$f_{\id}=\id_{S^1}$.
 \item $T$ has a positive real eigenvalue if and only if $f_T$ has a fixed point if and only if
$\rot(f_T)=0$.  Therefore $T$ has a negative real eigenvalue implies $\rot(f_T)\in\{0,1/2\}$.
\end{enumerate}

\begin{proof}[Proof of lemma \ref{L:fixed_point_of_pseudorotation_is_strongly_elliptic}]
By lemma \ref{L:appendix_blowing_up_fixedpt} we may conjugate the map
$\varphi:D\backslash\{0\}\rightarrow D\backslash\{0\}$ via a smooth orientation preserving
change of coordinates
$\pi:(0,1]\times\partial D\rightarrow D\backslash\{0\}$ to obtain an area preserving map
on the half open annulus that extends to a homeomorphism on the closed annulus
\[
	      \varphihat:[0,1]\times\partial D\rightarrow[0,1]\times\partial D.
\]
Moreover, on the ``new'' boundary $\{0\}\times\partial D$ the extended map is the circle map
induced by the differential of $\varphi$ at the removed fixed point.  That is,
\[
	\varphihat(0,x)=\left(0\,,\frac{D\varphi(0)[x]}{|D\varphi(0)[x]|}\right)
\]
for all $x\in\partial D$.  The annulus map $\varphihat$ preserves Lebesgue measure and is homotopic
to the identity and so the Poincar\'e-Birkhoff fixed point theorem applies.
That is, since $\varphi$ has no periodic points on $D\backslash\{0\}$ the map $\varphihat$ has no
periodic points in $(0,1]\times\partial D$, so the restriction of
$\varphihat$ to its two boundary components give circle homeomorphisms with equal rotation numbers.
In terms of $\varphi$ this is
\[
		  \rot(f_{D\varphi(0)})=\rot(\varphi|_{\partial D})
\]
where $f_{D\varphi(0)}$ is the circle map determined by the linear map
$T=D\varphi(0):\R^2\rightarrow\R^2$ as described in (\ref{E:9841}) above.
As $\varphi|_{\partial D}$ has no periodic points, its rotation number
\[
		  [\alpha]:=\rot(f_{D\varphi(0)})\in\R/\Z
\]
is irrational.  Therefore, $\rot(f_{D\varphi(0)})$ is irrational and so $D\varphi(0)$ cannot have
any real eigenvalues from the discussion immediately following (\ref{E:9841}) above.
Hence, being symplectic, the eigenvalues of $D\varphi(0)$ must be of the form
\begin{equation}\label{E:eigenvalues_at_origin}
 \sigma(D\varphi(0))=\{e^{2\pi i\theta}, e^{-2\pi i\theta}\}
\end{equation}
for some $\theta\notin(\Z\cup1/2\Z)$.  From the real canonical form,
$D\varphi(0)$ is therefore similar to the rotation through $2\pi\theta$ and also to the
rotation through $-2\pi\theta$.  Since these two rotations are themselves conjugate to each
other via an orientation reversing linear map, we conclude that
there always exists an \emph{orientation preserving} real linear map
$P:\R^2\rightarrow\R^2$ such that
\[
	P^{-1}D\varphi(0)P=R_{2\pi\theta}\qquad\mbox{or}\qquad P^{-1}D\varphi(0)P=R_{-2\pi\theta}.
\]
Therefore, $[\alpha]=\rot(f_{D\varphi(0)})=\rot(f_{R_{\pm2\pi\theta}})=\pm[\theta]$.  Meaning that
$[\alpha]=[\theta]$ or $[\alpha]=-[\theta]$.
So by (\ref{E:eigenvalues_at_origin}) the eigenvalues of $D\varphi(0)$ are
$\{e^{2\pi i\alpha}, e^{-2\pi i\alpha}\}$.
\end{proof}

\end{document}